\documentclass[a4paper,10pt]{amsart}
\usepackage[margin=1in]{geometry}
\usepackage{enumerate, amsmath, amsfonts, amssymb, amsthm, mathtools, thmtools, wasysym, graphics, graphicx, xcolor, frcursive, xparse, comment, ytableau, stmaryrd, bbm, array, colortbl, tensor, arydshln, leftidx, mathrsfs, hyperref, cleveref, manfnt, array, caption, thm-restate}
\definecolor{darkblue}{rgb}{0.0,0,0.7} 

\usepackage{tikz}
\usetikzlibrary{arrows.meta}
\usetikzlibrary{calc}
\usetikzlibrary{shapes.geometric}
\usetikzlibrary{fit}
\usetikzlibrary{backgrounds}

\newcommand\Matching[2]{%
  \begin{tikzpicture}[baseline=-0.65ex,scale=0.5, 
  every node/.style={text height=1.5ex, text depth=0.25ex}]
    \foreach \x [count=\xi] in {a,...,z}{
       \draw[circle,fill] (\xi,0)circle[radius=1mm]node[below]{$\x$};
        \ifnum \xi=#1
            \breakforeach
        \fi
    }
    \foreach \x/\y in {#2} {
       \pgfmathsetmacro{\Radius}{\y/2-\x/2}
       \draw(\x,0) arc[radius=\Radius, start angle=180, end angle=0];
    }
  \end{tikzpicture}%
}

\tikzset{
  vtx/.style={circle,draw,line width=.7pt,fill=white,inner sep=0pt,minimum size=6.5pt},
  e/.style={line width=.7pt},
  mult/.style={midway,fill=white,inner sep=2pt,font=\small},
  tag/.style={font=\footnotesize}
}

\newtheorem{theorem}{Theorem}[section]
\newtheorem{proposition}[theorem]{Proposition}
\newtheorem{corollary}[theorem]{Corollary}
\newtheorem{lemma}[theorem]{Lemma}
\newtheorem{example}[theorem]{Example}
\newtheorem{conjecture}[theorem]{Conjecture}
\newtheorem{remark}[theorem]{Remark}

\newcommand{\QQ}{\mathbb Q}
\newcommand{\RR}{\mathbb R}

\newcommand{\ZZ}{\mathbb Z}

\newcommand{\perm}{\operatorname{Perm}}
\newcommand{\defn}[1]{\textcolor{darkblue}{\emph{#1}}}

\newcommand{\tope}{\mathsf{P}}
\newcommand{\lpts}[1]{\mathsf{L}(#1)} 
\newcommand{\cube}{\mbox{\mancube}}
\newcommand{\hilb}{\operatorname{Hilb}}
\newcommand{\suchthat}{\;|\;}

\newcommand{\gri}[1]{\operatorname{gr}I(#1)}
\newcommand{\image}{\operatorname{Im}}
\newcommand{\grFrob}[1]{\operatorname{grFrob}(#1)} 
\newcommand{\Frob}[1]{\operatorname{Frob}(#1)} 

\newcommand{\reg}{\operatorname{reg}}

\numberwithin{equation}{section}

\title{Graded Ehrhart Theory for Hypersimplices}

\author{Nathaniel Libman \and Weston Miller}

\begin{document}
\begin{abstract}
    We prove that the $q$-Ehrhart series of a hyperplane slice of a cube is a rational function with an explicit denominator that satisfies $q$-reciprocity, confirming a conjecture of Reiner and Rhoades for these polytopes. To do this, we find a generating set for the orbit harmonics ideal, which also yields the Hilbert series and graded Frobenius characteristic of the associated quotient. We further show that the harmonic algebra of a hypersimplex $\Delta$ is generated as an algebra by the harmonic space of $\Delta$ using structural results on 2-factors of regular multigraphs. In particular, the harmonic algebra is finitely generated, giving a second proof of rationality.
\end{abstract}

\maketitle

\section{Introduction}

\subsection{Ehrhart theory}

One of the main objects of study in Ehrhart theory is the \defn{Ehrhart series} of a lattice polytope $\tope$. 
This power series is defined as 
\[
    E_\tope(t) = \sum_{m \geq 0} i_\tope(m) t^m,
\]
where $i_\tope(m)$ is the number of lattice points in the $m$th dilate of $\tope$, denoted $m\tope = \{mp : p \in \tope\}$.
The original work of Ehrhart \cite{ehrhart} shows that $E_\tope(t)$ is a rational function in $t$, for any lattice polytope $\tope$. There is an analogous series
\[
    \overline{E}_\tope(t) = \sum_{m \geq 0} \overline{i}_\tope(m) t^m,
\]
where $\overline{i}_\tope(m)$ is the number of lattice points in the relative interior of the $m$th dilate of $\tope$, denoted $\operatorname{int}(m\tope)$. This series is also a rational function and is determined by $E_\tope(t)$ via \defn{Ehrhart-Macdonald reciprocity}:
\[
    \overline{E}_{\tope}(t) = (-1)^{d+1}E_\tope(1/t),
\]
where $d$ is the dimension of $\tope$.

The Ehrhart series can be realized as the Hilbert series of the associated \defn{affine semigroup ring} $A_\tope = \mathbb{Q}[\mathbb{Z}^{d+1} \cap \operatorname{cone}(\tope)]$. The rationality of $E_\tope(t)$ follows from $A_\tope$ being a finitely generated $\mathbb{Q}$-algebra, and reciprocity follows from $A_\tope$ being Cohen-Macaulay and the \defn{interior ideal} $\overline{A}_\tope = \mathbb{Q}[\mathbb{Z}^{d+1} \cap \operatorname{int}\operatorname{cone}(\tope)]$ being the canonical module of $A_\tope$. It follows from the definition of $A_\tope$ that $A_\tope$ is generated in degree 1 when $\tope$ has the \defn{integer decomposition property} (IDP), that is when
\[
    \mathbb{Z}^d \cap m\tope = \underbrace{(\mathbb{Z}^d \cap \tope) + \cdots + (\mathbb{Z}^d \cap \tope)}_{m \text{ times}}.
\]

In \cite{reinerrhoades}, the authors define a \defn{$q$-Ehrhart series}
\begin{equation} \label{eq:q-ehr}
    E_\tope(t, q) = \sum_{m \geq 0} i_{\tope}(m;q) \cdot t^m,
\end{equation}
and
\[
    \overline{E}_\tope(t,q) = \sum_{m \geq 0} \overline{i}_{\tope}(m;q) \cdot t^m
\]
where $i_\tope(m;q)$ (respectively $\overline{i}_\tope(m;q)$) is the Hilbert series of the orbit harmonics quotient corresponding to the lattice points in the $m$th dilate of $\tope$ (respectively in the relative interior of the $m$th dilate of $\tope$). In particular $i_\tope(m;1) = i_\tope(m)$ (respectively $\overline{i}_\tope(m;1) = \overline{i}_\tope(m)$). 

The \textit{orbit harmonics quotient} of a finite set of points $Z$ is the graded ring $R(Z) := S/\gri{Z}$, where 
\[
    \gri{Z} := \langle \text{top degree homogeneous component of } f : f(z) = 0 \text{ for all } z \in Z\rangle
\]
is the \textit{orbit harmonics ideal}. This ring has $\dim R(Z) = |Z|$ and will be explored more in \Cref{sec:back}.

Reiner and Rhoades make the following conjecture for $q$-Ehrhart series:
\begin{conjecture}[\cite{reinerrhoades}, Conjecture 1.1]
\label{conj:series}
Let $\tope$ be a $d$-dimensional lattice polytope in $\mathbb{R}^n$. Then
\begin{enumerate}
    \item[(i)] $E_\tope(t,q)$ and $\overline{E}_\tope(t,q)$ are rational functions in $t$ and $q$ expressible over the same denominator $\prod_{i=1}^\nu (1-q^{a_i}t^{b_i})$ with $\nu \geq d + 1$ and numerators in $\mathbb{Z}[t,q]$.
    \item[(ii)] If $\tope$ is a lattice simplex and $\nu = d + 1$, then both numerators have nonnegative coefficients as polynomials in $t,q$.
    \item[(iii)] There is a $q$-analogue of Ehrhart-Macdonald reciprocity:
    \[
        q^d \overline{E}_\tope(t,q) = (-1)^{d+1}E_\tope(1/t,1/q).
    \]
\end{enumerate}
\end{conjecture}

In \cite{reinerrhoades} the authors define the \textit{harmonic algebra} $\mathcal{H}_\tope$ using Macaulay inverse systems. This algebra is an analogue of the affine semigroup ring $A_\tope$ and will be precisely defined in \Cref{sec:back}. The bigraded Hilbert series of $\mathcal{H}_\tope$ is the $q$-Ehrhart series $E_\tope(t,q)$, and the rationality of $E_\tope(t,q)$ and $q$-reciprocity are implied by algebraic properties of $\mathcal{H}_\tope$.

\begin{conjecture}[\cite{reinerrhoades}, Conjecture 5.5]
\label{conj:alg}
     For a lattice polytope $\tope$, the harmonic algebra $\mathcal{H}_\tope$ is
     \begin{enumerate}
         \item[(i)] a finitely generated $\mathbb{Q}$-subalgebra of $\mathbb{Q}[x_0,x_1,\dots,x_n]$,
         \item[(ii)] a Cohen-Macaulay algebra, and
         \item[(iii)] its canonical module is isomorphic to $\overline{\mathcal{H}}_\tope$, up to a shift in grading.
     \end{enumerate}
\end{conjecture}

\subsection{The state of the conjecture}

Part (i) of \Cref{conj:alg} was disproved by Cavey \cite{cavey2025graded} for the lattice triangle with vertices $(0,0),(7,56),(-45,30)$ using connections to geometry and Mori dream spaces. No counterexample to \Cref{conj:series} has appeared in the literature, but the algebraic approach suggested by \Cref{conj:alg} will not work for all polytopes. 

Reiner and Rhoades prove \Cref{conj:series} for all one-dimensional lattice polytopes, antiblocking polytopes, standard simplices, cross-polytopes, cubes, order polytopes, and chain polytopes of finite posets \cite{reinerrhoades}. Recently, Crowley and Partida proved \Cref{conj:alg} for unimodular zonotopes \cite{crowley2026graded}.

\subsection{Why hypersimplices?}

For positive integers $\ell,n$ with $\ell \leq n$, the \defn{$(\ell,n)$-hypersimplex} is
\[
    \Delta_{\ell,n} := \left\{x \in [0,1]^n : \sum_{i=1}^n x_i = \ell \right\}.
\]
The Ehrhart polynomial of $\Delta_{\ell,n}$ was originally found by Katzman in \cite{Katzman01032005}:
\[
    i_{\Delta_{\ell,n}}(m) = \sum_{j=0}^{\ell-1} (-1)^j \binom{n}{j} \binom{(\ell - j)m - j + n-1}{n-1}.
\]
Since then, hypersimplices have been a central object of study in Ehrhart theory. In \cite{early2017conjectures}, Early conjectured a combinatorial interpretation of the $h^*$-vector of the hypersimplex using decorated ordered set partitions. This conjecture was proven by Kim in \cite{KIM2020105213}. In \cite{FERRONI2021105365}, Ferroni proved that hypersimplices are Ehrhart positive, settling a case of \cite[Conjecture 2]{DeLoera2009}. In \cite{Clarke_Kolbl_2025}, the authors studied the equivariant Ehrhart theory of hypersimplices, including an interpretation in terms of decorated ordered set partitions under the action of $\mathfrak{S}_n$.

Unlike general lattice polytopes, hypersimplices are stable under the $\mathfrak{S}_n$-action permuting coordinates. So the orbit harmonics quotient of each dilate is a graded $\mathfrak{S}_n$-module, and assembling these modules makes the harmonic algebra a bigraded $\mathfrak{S}_n$-algebra. So for hypersimplices, the coefficients of the $q$-Ehrhart series upgrade from integers to representations of $\mathfrak{S}_n$. In \Cref{sec:frob}, we compute the graded Frobenius characteristic of the orbit harmonics quotient for each dilate. 

For much of the paper, we will consider the larger family of cube slices 
\[
    \tope_{\ell,m,n} := \left\{z \in [0,m]^n \mid z_1 + \cdots + z_n = \ell \right\}.
\]
Dilations and relative interiors of hypersimplices are inside this family. That is, 
\[ 
    m\Delta_{\ell,n} = \tope_{\ell m,m,n} \quad \text{and} \quad \operatorname{int}(m\Delta_{\ell,n}) \cap \mathbb{Z}^n \cong \tope_{\ell m - n,m-2,n} \cap \mathbb{Z}^n.
\]
Since $\tope_{\ell,m,n}$ and $\tope_{nm - \ell,m,n}$ are affine equivalent, we often assume that $\ell \leq \ell_{\max} = \lfloor nm/2 \rfloor$.

\subsection{Main results}

In this paper, we prove \Cref{conj:series} parts (i) and (iii) for cube slices and state an explicit denominator for the $q$-Ehrhart series.

\begin{restatable}{introtheorem}{thmrational}
\label{thm:rational}
    For $1 \leq \ell \leq \ell_{\max}$, $E_{\tope_{\ell,m,n}}(t,q)$ is rational with denominator 
    \[
        (1-q^\ell t)^{n-1} \prod_{j=0}^{\lceil \ell/m \rceil -1} (1-q^{jm}t),
    \]
    numerator in $\mathbb{Z}[t,q]$, and $q^{n-1}\overline{E}_{\tope_{\ell,m,n}}(t,q) = (-1)^n E_{\tope_{\ell,m,n}}(1/t,1/q)$. If $n \geq 3$ and $2\ell = nm$, the exponent $n-1$ on $(1-q^\ell t)$ improves to $n-2$.
\end{restatable}

Starting from an explicit form of the $q$-Ehrhart series, the proof is elementary and only involves re-indexing sums and generating function techniques. Part (ii) of \Cref{conj:series} is vacuous unless $\tope_{\ell,m,n}$ is a simplex, which only happens for $\ell \leq m$, in which case $\tope_{\ell,m,n} = \ell \Delta_{1,n}$. Reiner and Rhoades already handled $\Delta_{1,n}$ explicitly \cite[Proposition 3.27]{reinerrhoades}. Moreover, $\ell \Delta_{1,n}$ is affine equivalent to the antiblocking polytope $\ell \operatorname{Pyr}(\Delta_{n-2}) = \{y \geq 0 : \sum_{i < n} y_i \leq \ell, \, y_n = 0\}$, which Reiner and Rhoades showed satisfies \Cref{conj:series} \cite[Corollary 3.17]{reinerrhoades}.

We also prove that the harmonic algebra for hypersimplices is generated in degree 1. This fact proves part (i) of \Cref{conj:alg} for hypersimplices, giving an alternate proof of rationality.

\begin{restatable}{introtheorem}{thmgendegone}
\label{thm:gendegone}
    $\mathcal{H}_{\Delta_{\ell,n}}$ is generated as an algebra in $x_0$-degree 1. In particular, it is finitely generated.
\end{restatable}

The harmonic space $V_{m\Delta_{\ell,n}}$ (which will be defined in \Cref{sec:back}) is spanned by 
\[
    \delta_G := \prod_{(i,j)\in G} (x_i - x_j)
\]
for loopless multigraphs $G$ on vertex set $[n]$ with $\leq \ell m$ edges and degrees $\leq m$. The set of all such multigraphs is denoted $G(\ell,m,n)$. Multiplication in $\mathcal{H}_{\Delta_{\ell,n}}$ is the union of edge multisets. The degree-1 generation of $\mathcal{H}_{\Delta_{\ell,n}}$ translates into a structural graph theory problem. This problem is to show that every element of $G(\ell,m,n)$ is a linear combination of multigraphs that can be factored into $m$ graphs each with $\leq \ell$ edges and degrees $\leq 1$.

This result is analogous to $A_{\Delta_{\ell,n}}$ being generated in degree 1, which holds for hypersimplices because they have the IDP. In general, satisfying the IDP is not enough to guarantee that $\mathcal{H}_{\tope}$ is generated in $x_0$-degree 1. A counterexample is the triangle $\operatorname{conv}\{(1,0),(0,1),(-1,-1)\}$. Computational evidence suggests that $\mathcal{H}_{\tope_{\ell,m,n}}$ is also finitely generated in $x_0$-degree 1, but many of the steps in our approach do not immediately extend to this case. 

The driving result of our paper is the explicit computation of the orbit harmonics ideal of $Z_{\ell,m,n} = \tope_{\ell,m,n} \cap \mathbb{Z}^n$.

\begin{restatable}{introtheorem}{thmgri} \label{thm:gri}
  For $\ell \leq \ell_{\max}$ we have
  \[
      \gri{Z_{\ell,m,n}}=\langle x_1^{m+1},\dots,x_n^{m+1},x_1+\cdots+x_n, x_1^{a_1}\cdots x_n^{a_n}\text{ for } a_1+\cdots+a_n=\ell+1\rangle.
  \]
\end{restatable}

This ideal is as simple as one could hope: the relations stem from the cube, the linear form, and a degree truncation corresponding to the precise slice. We will prove this theorem with a dimension count invoking the weak Lefschetz property of the monomial complete intersection $A_{n,m} = \mathbb{Q}[x_1,\dots,x_n]/\langle x_1^{m+1},\dots,x_n^{m+1} \rangle$.

In the appendix, we develop the graph theory further to give a complete characterization of graph factorization at the level of sets. 
\begin{restatable}{introtheorem}{thmfullfact}
\label{thm:fullfact}
    For $m,m' \geq 1$ the containment
    \[
        G(\ell,m,n) \cdot G(\ell,m',n) \subseteq G(\ell,m + m' ,n)
    \]
    is an equality iff
    \begin{enumerate}
        \item $m$ and $m'$ are both even,
        \item $m$ and $m'$ are both odd and $\hat\ell \leq 1$,
        \item $m$ is odd, $m'$ is even and $m \geq \frac{2\hat\ell - 3}{3}$ (up to swapping $m$ and $m'$),
    \end{enumerate}
    where $\hat\ell = \min\{\ell,\lceil n/2 \rceil \}$.
\end{restatable}

Here,
\[
    G(\ell,m,n) \cdot G(\ell,m',n) := \{G_1 \cdot G_2 : G_1 \in G(\ell,m,n) \text{ and } G_2 \in G(\ell,m',n)\}.
\]
The proof is purely graph-theoretic, and involves a long analysis of the structure of the graphs' bridge forests.

\subsection{Outline}

In \Cref{sec:back}, we provide some background on orbit harmonics and symmetric functions. In \Cref{sec:frob}, we use the graded Frobenius characteristic to describe the orbit harmonics of slices of a cube. This description will allow us to write the $q$-Ehrhart series of a slice of a cube. In \Cref{sec:ratRec}, we directly prove that rationality and $q$-reciprocity are satisfied for slices of the cube. In \Cref{sec:harm}, we investigate the harmonic algebra for hypersimplices and show that it is finitely generated in $x_0$-degree 1 using connections to structural graph theory.

\section*{Acknowledgments}

We thank Brendon Rhoades for suggesting this problem and providing guidance throughout. We thank Vasu Tewari for being involved in an early draft of this paper.

\section*{Use of AI}

The proof of \Cref{thm:rational} was developed through conversation with Claude Opus 4.8 and Opus 5. The idea to prove \Cref{cor:path} via \Cref{lem:bipartite} was entirely due to Claude Opus 5, though the actual proof of \Cref{lem:bipartite} was developed primarily by the authors. The authors also used AI models to generate some examples and TikZ figures, as well as for general revision of the paper. The authors take full responsibility for the accuracy of all mathematical content.

\section{Background}
\label{sec:back}

\subsection{Orbit harmonics}

Throughout the paper we fix a positive integer $n$ and let $S \coloneqq \QQ[x_1, \dots, x_n]$ denote the polynomial ring in $n$ variables with its standard grading $\deg(x_i) = 1$ for all $i$. 
Given a graded vector space $V = \bigoplus_{d \geq 0} V_d$, we can write its \defn{Hilbert series}
\[
    \hilb(V; q) = \sum_{d \geq 0} (\dim(V_d)) \cdot q^d.
\]
For a bigraded vector space $V = \bigoplus_{i, j \geq 0} V_{i, j}$, its bigraded Hilbert series is
\[
    \hilb(V; t, q) = \sum_{i, j \geq 0} (\dim(V_{i, j})) \cdot t^i q^j.
\]

Given a nonzero polynomial $f \in S$, we define $\tau(f)$ to be its top degree homogeneous component.
Let $Z \subset \RR^n$ be a finite point locus. 
We can associate to $Z$ two ideals: the vanishing ideal $I(Z)$ and its associated graded ideal $\gri{Z}$, which we refer to as its \defn{orbit harmonics ideal}.
These are denoted 
\[
    I(Z) \coloneqq \{ f(x_1, \dots, x_n) \in S: f(z) = 0 \text{ for all } z \in Z \}
\]
and 
\[
    \gri{Z} = \langle \tau(f): f \in I(Z) \rangle,
\]
respectively.
We define the \defn{orbit harmonics quotient} of $Z$ to be $R(Z) \coloneqq S / \gri{Z}$. We have that
\begin{equation}\label{eq:dim}
    \dim R(Z) = |Z|,
\end{equation}
and since $\gri{Z}$ is homogeneous, the quotient $R(Z)$ inherits a grading from $S$. If $Z$ possesses an action of a symmetric group $\mathfrak{S}_n$, then $R(Z) \cong \QQ[Z]$ as ungraded $\mathfrak{S}_n$ representations. We denote by $\operatorname{Aff}(\mathbb{Z}^n) = GL_n(\mathbb{Z}) \ltimes \mathbb{Z}^n$ the group of \defn{affine transformations}. We say that point loci $Z, Z'$ are \defn{affine equivalent} if there exists $g \in \operatorname{Aff}(\mathbb{Z}^n)$ such that $Z' = gZ$. In this case, $R(Z) \cong R(Z')$ as graded algebras. We will say that polytopes $\tope, \tope'$ are affine equivalent if $\lpts{\tope}, \lpts{\tope'}$ are affine equivalent. In this case, we will write $\tope \cong \tope'$. The graded Ehrhart data of a polytope only depends on the affine equivalence class \cite[Proposition 3.3]{reinerrhoades}.

We define a pairing $\odot: S \times S \to S$ by 
\[
    f \odot g \coloneqq f \left( \frac{\partial}{\partial x_1}, \dots, \frac{\partial}{\partial x_n} \right) g.
\]
In other words, we let $S$ act on itself by taking partial derivatives. This allows us to define a bilinear pairing 
\[
    \langle f, g \rangle \coloneqq \text{The constant term of } f \odot g.
\]

Given a homogeneous ideal $I \subset S$, we define the \defn{harmonic space} (or \textit{Macaulay inverse system}) associated to that ideal as
\[
    I^\perp \coloneqq \{g(x) \in S: \langle f, g \rangle = 0 \text{ for all } f \in I\}.
\]
One can show (see e.g. \cite{reinerrhoades}) that $I^\perp = \{g(x) \in S: f \odot g = 0 \text{ for all } f \in I \}$, and that 
\[
    \hilb(I^\perp; q) = \hilb(S/I; q).
\]

Given a point locus $Z$, we define 
\[
    V_Z \coloneqq (\gri{Z})^\perp,
\]
and refer to this space as the \textit{harmonic space of $Z$}. By \Cref{eq:dim}, we see that $\hilb(V_Z; q)$ provides a natural $q$-analogue of $|Z|$. 
Given a set $A\subset \RR^n$, we let $\lpts{A}$ denote the set of lattice points in $A$, i.e. $\lpts{A}=A\cap \ZZ^n$.
We define
\[
    V_\tope := V_{\lpts{\tope}},
\]
for $\tope \subset \mathbb{R}^n$ a polytope. This vector space will be referred to as the \textit{harmonic space of $\tope$}.

We now review the definition of the \defn{harmonic algebra} associated to a lattice polytope $\tope$. 
See \cite{reinerrhoades} for more details and examples.
Let $x_0$ denote an indeterminate, and define the harmonic algebra of a polytope $\tope \subset \RR^n$ to be
\[
    \mathcal{H}_\tope = \bigoplus_{m \geq 0} \QQ \cdot x_0^m \otimes V_{m\tope}.
\]
It is not a priori obvious that $\mathcal{H}_\tope$ is closed under multiplication. This property is proved in \cite[Proposition 5.4]{reinerrhoades}. We consider the polynomial ring $\QQ[x_0, x_1, \dots, x_n]$ to be bigraded with $\deg(x_0) = (1, 0)$ and $\deg(x_i) = (0, 1)$ for $1 \leq i \leq n$. One can check that the $q$-Ehrhart series defined in \Cref{eq:q-ehr} is indeed the bigraded Hilbert series of $\mathcal{H}_\tope$. A crystallographic linear group action stabilizing $\lpts{\tope}$ induces an action of bigraded algebra automorphisms fixing $x_0$ on $\mathcal{H}_\tope$.

\subsection{Symmetric functions}
We now give a brief review of symmetric functions and representation theory of $\mathfrak{S}_n$. 
See \cite{macdonald} for a more detailed overview of these concepts.
Define a \defn{partition} $\lambda = (\lambda_1, \dots, \lambda_l)$ to be a weakly decreasing sequence of nonnegative integers. 
The length $\ell(\lambda)$ is the number of \emph{positive} parts of $\lambda$.
A \defn{Young diagram} of shape $\lambda$ is an upper-left justified sequence of empty boxes, with $\lambda_i$ boxes in the $i^{th}$ row. 
A \defn{Young tableau} of shape $\lambda$ is a filling of these boxes with integers.
Such a filling is called \defn{semistandard} if the values weakly increase across rows and strictly increase down columns.
The set of semistandard Young tableaux of shape $\lambda$ is denoted $\text{SSYT}(\lambda)$.

A \defn{symmetric function} over $\QQ$ is a formal power series of bounded degree in infinitely many variables with rational coefficients which is invariant under any permutation of the variables. 
The algebra of all symmetric functions over $\QQ$ is denoted $\Lambda_\QQ$.
An important symmetric function is the complete homogeneous symmetric function $h_\lambda$, defined as $h_\lambda \coloneqq h_{\lambda_1}\dots h_{\lambda_l}$, and
\[
    h_r \coloneqq \sum_{1 \leq i_1 \leq \dots \leq i_r} x_{i_1}\dots x_{i_r}.
\]

Given a partition $\lambda$, the \defn{Schur function} $s_\lambda$ is defined by
\[
    s_\lambda(x_1, x_2, \dots) \coloneqq \sum_{T \in \text{SSYT}(\lambda)} \prod_{i \geq 1} x_i^{\# \text{ of } i's \text{ in } T} .
\]
The $s_\lambda$ are symmetric (\cite{sagan}) and the collection $\{s_\lambda: \lambda \text{ a partition of } n\}$ forms a $\QQ$-linear basis for the vector space of symmetric functions of degree $n$.

The irreducible complex representations of $\mathfrak{S}_n$ are indexed by partitions $\lambda$ of $n$. The \defn{Frobenius characteristic} is a map 
\[
    \text{Frob}: C(\mathfrak{S}_n) \to \Lambda_\QQ^{(n)}
\]
from the vector space of class functions $\mathfrak{S}_n \to \mathbb{Q}$ to the vector space of homogeneous symmetric functions over $\mathbb{Q}$ of degree $n$. It is defined as the linear extension of
\[
    \text{Frob}(\chi^{\lambda}) \coloneqq s_\lambda
\]
where $\chi^\lambda$ is the irreducible character of $\mathfrak{S}_n$ corresponding to the partition $\lambda$. Since the character of any representation of $\mathfrak{S}_n$ is a class function, this map gives a way to associate a symmetric function to any representation of $\mathfrak{S}_n$. 
Given a representation $V$, we write $\Frob{V}$ for the symmetric function corresponding to the character of $V$. If $V = \bigoplus_{d \geq 0} V_d$ is a graded representation of $\mathfrak{S}_n$, then the \defn{graded Frobenius characteristic} $\grFrob{V}$ is defined as 
\[
    \grFrob{V} \coloneqq \sum_{d \geq 0} \Frob{V_d} \cdot q^d.
\]

\section{Orbit harmonics of cube slices}\label{sec:frob}

Some of the results in this section appeared in the first author's doctoral thesis \cite{libman2025schur}. For $\lambda=(\lambda_1\geq \cdots \geq  \lambda_n) \in \RR^n$ we define $\perm_{\lambda}$ to be the \defn{permutahedron} determined by $\lambda$: the convex hull of the points $\{\lambda_{\sigma}\}_{\sigma\in \mathfrak{S}_n}$ where $\lambda_{\sigma}\coloneqq (\lambda_{\sigma(1)},\dots,\lambda_{\sigma(n)})$.
We have that $\mathfrak{S}_n$ acts on the set $\lpts{\perm_{\lambda}}$ by permuting coordinates and thus gives a permutation representation of dimension $|\lpts{\perm_{\lambda}}|$.
The method of orbit harmonics lifts this representation to a graded $\mathfrak{S}_n$ representation.

\subsection{The algebras $A_{n,m}$ and $B_{n,m}$}

Given a positive integer $m$, let $m\cube_n$ denote the cube $[0,m]^n$. We have $|\lpts{m\cube_n}|= (m+1)^n$ and
\[
    I_{n,m} := \gri{\lpts{m\cube_n}} = \langle x_1^{m+1},\dots,x_n^{m+1} \rangle
\]
by \cite[Lemma 3.11]{reinerrhoades}. The orbit harmonics quotient is
\[
  A_{n,m}:=\QQ[x_1,\dots,x_n]/ I_{n,m}.
\]
Let $A_{n,m,d}$ denote the degree $d$ graded piece of $A_{n,m}$. The Hilbert series of $A_{n,m}$ is easily seen to be
\[
    \hilb(A_{n,m};q)\coloneqq \sum_{d\geq 0}\dim_{\QQ}(A_{n,m,d})q^d=\left(1+q+\dots + q^m\right)^n.
\]
The polynomial $\left(1+q+\dots + q^m\right)^n$ is symmetric and unimodal of degree $nm$. We define $\ell_{\max} = \lfloor nm /2 \rfloor$, which is the degree at which this polynomial has its maximal coefficient. 

Let $c(d,m,n):= \dim A_{n,m,d}$ with the convention that $A_{n,m,d} = 0$ if $d < 0$ or $m < 0$. We have that $c(d,m,n)$ is equal to the number of weak compositions of $d$ (that is, tuples of nonnegative integers summing to $d$) into $n$ parts each of size $\leq m$. Alternatively, $c(d,m,n)$ can be viewed as the number of lattice points in the cube $m\cube_n$ with coordinate sum $d$. We define $\delta(d,m,n) := c(d,m,n) - c(d-1,m,n)$. Oh and Rhoades express the graded Frobenius characteristic of $A_{n,m}$ as an $h$-positive symmetric function as follows:

\begin{proposition} (\cite{OR22}, Proposition 1) \label{OR22}
    \begin{align*}
        \grFrob{A_{n,m}} = \sum_{\substack{\lambda_1 < m+1 \\ \ell(\lambda) \leq n}} q^{|\lambda|} \cdot h_{\text{mult}(\lambda)},
    \end{align*}
    where $\text{mult}(\lambda)$ is the weakly decreasing rearrangement of $(n-\ell(\lambda), \text{mult}_1(\lambda), \text{mult}_2(\lambda), \dots, \text{mult}_{m}(\lambda))$, and $\text{mult}_i(\lambda)$ is the multiplicity of the value $i$ in $\lambda$.
\end{proposition}

Next, we consider a quotient of $A_{n,m}$ by the linear form $x_1 + \cdots + x_n$. This algebra will play an important role in the orbit harmonics for slices of the cube by the hyperplane $x_1 + \cdots + x_n = \ell$. The algebra $B_{n,m}$ is defined as 
\begin{align*}
  B_{n,m}=S/ {\widetilde{I}_{n,m}}
\end{align*}
where $\widetilde{I}_{n,m}\coloneqq \langle x_1^{m+1},\dots,x_n^{m+1},x_1+\cdots+x_n \rangle$.

\begin{theorem} \label{thm grFrobB}
For $B_{n,m}$, we have
\begin{enumerate}
    \item $\dim B_{n,m,d} = \delta(d,m,n)$ for $d \leq \ell_{max}$, and $B_{n,m,d} = 0$ for $d > \ell_{max}$.

    \item $\grFrob{B_{n,m}}= \sum_{d\leq \ell_{\max}} (\Frob{A_{n,m,d}}-\Frob{A_{n,m,d-1}})q^d$.
\end{enumerate}
\end{theorem}
\begin{proof}
  Begin by noting that $A_{n,m}$ is a monomial complete intersection, so in characteristic zero $L\coloneqq x_1+\cdots+x_n$ is a strong Lefschetz element \cite{stanleyLefschetz}. We only will only use the resulting weak Lefschetz property of $L$.
  This property implies that the multiplication-by-$L$ map has maximal rank in each degree. Combining this fact with the unimodality of $\hilb(A_{n,m};q)$, we get that
 \begin{align*}
   \phi_d:A_{n,m,d-1}\to A_{n,m,d}
 \end{align*}
 defined by $\phi_d(f)=L\cdot f$ is an injection for $1\leq d\leq \ell_{\max}$ and a surjection for $d>\ell_{\max}$. We have that the graded degree $d>0$ piece $B_{n,m,d}$ satisfies
 \begin{align*}
   B_{n,m,d} = A_{n,m,d} / \image(\phi_d).
 \end{align*}
 So $\dim B_{n,m,d} = \delta(d,m,n)$ for $d \leq \ell_{max}$, and $B_{n,m,d} = 0$ for $d > \ell_{max}$. Hence $\hilb(B_{n,m};q)$ is the truncation of $(1-q)\hilb(A_{n,m};q)$ to degrees $\leq \ell_{\max}$.
 
 Since $L$ is $\mathfrak{S}_n$-invariant, $\phi_d$ is an $\mathfrak{S}_n$-homomorphism.
 So the graded Frobenius characteristic of the $\mathfrak{S}_n$-module $B_{n,m}$ is
\begin{align*}
  \grFrob{B_{n,m}}=\sum_{d\geq 0}\Frob{B_{n,m,d}}q^d=\sum_{d\leq \ell_{\max}} (\Frob{A_{n,m,d}}-\Frob{A_{n,m,d-1}})q^d.
\end{align*}
\end{proof}

A symmetric function is \defn{Schur-positive} if its expansion in the Schur basis has nonnegative coefficients, or equivalently, if it is the Frobenius image of an $\mathfrak{S}_n$-representation.

\begin{corollary}
For $0 \leq d \leq \ell_{\max}$, $\Frob{A_{n,m,d}}-\Frob{A_{n,m,d-1}}$ is Schur-positive.
\end{corollary}

\subsection{The slices $\tope_{\ell,m,n}$}

We will be interested in slices of $m\cube_{n}$ given by intersecting with translates of the hyperplane $H\coloneqq \{(z_1,\dots,z_n)\suchthat z_1+\cdots+z_n=0\}$.
We only need to consider slices
\begin{align} \label{eq:tope}
    \tope_{\ell,m,n} \coloneqq m\cube_n\cap H^{(\ell)},
\end{align}
where $H^{(\ell)}\coloneqq \{(z_1,\dots,z_n)\suchthat z_1+\cdots+z_n=\ell\}$ and $\ell$ satisfies $0 < \ell <nm$. The dimension of the polytope $\tope_{\ell,m,n}$ is $n-1$.
The polytopes $\tope_{\ell,m,n}$ are permutahedra: let $a,b$ be the unique nonnegative integers such that $\ell=am+b$, with $b < m$.
Then we have that
\begin{align*}
    \tope_{\ell,m,n}=\perm_{(m^a,b,0^{n-a-1})}.
\end{align*} 
The slices $\tope_{\ell,m,n}$ and $\tope_{nm-\ell,m,n}$ are affine equivalent, so we may restrict our attention to the range $0 < \ell\leq \ell_{\max}$.

Let $Z_{\ell,m,n}=\lpts{\tope_{\ell,m,n}}$ so that $|Z_{\ell,m,n}| = c(\ell,m,n)$. A dilation of a slice of a cube is a slice of a larger cube. More precisely, we have $k\tope_{\ell,m,n} = \tope_{k\ell,km,n}$. The set of relative interior lattice points of a slice of a cube is affine equivalent to the set of lattice points in the slice of a smaller cube. More precisely, we have $\operatorname{int}(\tope_{\ell,m,n}) \cap \mathbb{Z}^n \cong Z_{\ell - n, m-2, n}$.

\subsection{The orbit harmonics ideal}

In this subsection, we will compute the orbit harmonics ideal of $Z_{\ell,m,n}$.
\begin{proposition}\label{pro:members_vanishing_ideal}
  The following polynomials lie in  $I(Z_{\ell,m,n})$.
  \begin{enumerate}
    \item[(i)] \label{it1} $x_1+\cdots+x_n-\ell$.
    \item[(ii)] \label{it2} $x_i(x_i-1)\cdots (x_i-m)$ for $1\leq i\leq n$.
    \item[(iii)] \label{it3} The polynomials
    \[
    f_a=\prod_{\substack{1\leq i\leq n\\a_i>0}}x_i(x_i-1)\cdots (x_i-a_i+1)
    \]
    where $a=(a_1,\dots,a_n)\in \ZZ_{\geq 0}^n$ is such that $a_1+\cdots+a_n > \ell$.
  \end{enumerate}
\end{proposition}
\begin{proof} 
    Let $p\coloneqq (p_1,\dots,p_n)\in Z_{\ell,m,n}$.
    \begin{enumerate}
        \item[(i)] $x_1 + \cdots + x_n - \ell$ vanishes at $p$ because $p$ satisfies the hyperplane equation $p_1 + \cdots + p_n = \ell$.
        \item[(ii)] Each $p_i$ takes values in $\{0,\dots,m\}$ and hence $\prod_{0\leq j\leq m}(p_i-j)=0$, so the polynomial 
        \[
            x_i(x_i-1)\cdots (x_i-m)
        \]
        vanishes at $p$.
        \item[(iii)] By permuting coordinates, it suffices to consider the cases with $a=(a_1,\dots,a_r,0^{n-r})$ where $a_1,\dots,a_r>0$ (and $r$ is some positive integer). If any $p_i$ for $1\leq i\leq r$ satisfies $0\leq p_i<a_i$, then $f_a$ clearly vanishes. Hence suppose $p_i\geq a_i$ for $1\leq i\leq r$. But then $\sum_{i=1}^r p_i\leq \ell$ and $\sum_{i=1}^r a_i>\ell$, which contradicts $\sum_{i=1}^r p_i \geq \sum_{i=1}^r a_i$.
    \end{enumerate}
\end{proof}

The next theorem shows that the top-degree components of the polynomials in \Cref{pro:members_vanishing_ideal} generate the orbit harmonics ideal. This phenomenon does not always occur -- the top-degree components of a generating set of an ideal do not necessarily generate the associated graded ideal. A simple counterexample is the ideal $I = \langle x^2 + y, xy \rangle$ in $\mathbb{Q}[x,y]$. The top-degree components of these generators give $\langle x^2,xy \rangle$. But 
\[
    y^2 = y(x^2 + y) - x(xy),
\]
so $y^2 \in \operatorname{gr} I$, but $y^2 \not\in \langle x^2,xy \rangle$. The computation of the orbit harmonics ideal is a major part of the difficulty in much of orbit harmonics theory.

\thmgri*
\begin{proof}
    Let $J$ denote the ideal on the right-hand side of \Cref{thm:gri}. 
    Then by \Cref{pro:members_vanishing_ideal}, we have $J \subseteq \gri{Z_{\ell,m,n}}$.

    For the reverse, let $S_+$ be the ideal $\langle x_1,\dots,x_n \rangle$. Then, $J = \widetilde{I}_{n,m} + S_+^{\ell + 1}$. Hence $S/J$ is the degree $\leq \ell$ truncation of $B_{n,m}$. By \Cref{thm grFrobB}, we get
    \[
        \dim S/J = \sum_{d = 0}^\ell \dim B_{n,m,d} = \sum_{d= 0}^\ell \delta(d,m,n) = c(\ell,m,n) = |Z_{\ell,m,n}| = \dim S/\gri{Z_{\ell,m,n}}.
    \]
    Hence $J = \gri{Z_{\ell,m,n}}$.
\end{proof}

\begin{remark}
    At $\ell = \ell_{\max}$, we have $S_+^{\ell_{\max}+1} \subseteq \widetilde{I}_{n,m}$, so $\gri{Z_{\ell_{\max},m,n}} = \widetilde{I}_{n,m}$.
\end{remark}

\begin{remark}
    If $m \geq \ell$, the generators $x_i^{m+1}$ are redundant. In this case, we have that $\tope_{\ell,m,n} = \ell \Delta_{1,n}$ is a dilated simplex.
\end{remark}

\begin{corollary}
    $S / \gri{Z_{\ell,m,n}} \cong \bigoplus_{d \leq \ell} B_{n,m,d}$ as graded $\mathfrak{S}_n$-modules.
\end{corollary}

\begin{corollary}
    \[
        \grFrob{S/\gri{{Z_{\ell,m,n}}}} = \sum_{d\leq \ell} (\Frob{A_{n,m,d}}-\Frob{A_{n,m,d-1}})q^d
    \]
    and
    \[
        \hilb(S/\gri{Z_{\ell,m,n}}) = [(1-q)(1+q+\dots+q^m)^n]_{\leq \ell}= \sum_{d = 0}^{\ell} \delta(d,m,n) q^d.
    \]
\end{corollary}

\begin{corollary}
\label{cor:EhrTope}
For $\ell \leq \ell_{\max}$,
    \[
        E_{\tope_{\ell,m,n}}(t,q) = \sum_{k \geq 0} \sum_{d = 0}^{k \ell} \delta(d,km,n) q^{d}t^{k}.
    \]
\end{corollary}

\begin{corollary}
\label{cor:intEhrTope}
    For $\ell \leq \ell_{\max}$,
    \[
        \overline{E}_{\tope_{\ell,m,n}}(t,q) = \sum_{k \geq 0} \sum_{d = 0}^{k \ell - n} \delta(d,km - 2,n) q^{d}t^{k}.
    \]
\end{corollary}

\section{Rationality and Reciprocity}
\label{sec:ratRec}

In this section we will prove directly using generating function techniques that $E_{\tope_{\ell,m,n}}(t,q)$ is rational and satisfies $q$-reciprocity. We will use the following three lemmas in the proof of \Cref{thm:rational}. The first lemma gives us a method for showing the rationality of certain power series and a reciprocity result for such series.

\begin{lemma}
\label{lem:polyrat}
    Suppose $f(x_1,\dots,x_k) \in \mathbb{C}[x_1,\dots,x_k]$ with degree $d_i$ as a polynomial in $x_i$. Then 
    \[
        \sum_{m_1,\dots,m_k \geq 0} f(m_1,\dots,m_k) x_1^{m_1}\cdots x_k^{m_k}
    \]  
    is a rational function $F(x_1,\dots,x_k)$ with denominator dividing $\prod_{i=1}^k (1-x_i)^{d_i + 1}$. Moreover,
    \[
        F(1/x_1,\dots,1/x_k) = (-1)^k \sum_{m_1,\dots,m_k \geq 1} f(-m_1,\dots,-m_k) x_1^{m_1}\cdots x_k^{m_k}.
    \]
\end{lemma}
\begin{proof}
    For $k = 1$, this is Corollary 4.3.1 and Proposition 4.2.3 in \cite{Stanley_2011}. The univariate case then implies the statement for simple tensors in
    \[
        \{f(x_1,\dots,x_k) : \deg_{i} f \leq d_i \text{ for all } i\} = \bigotimes_{i=1}^k \{g(x_i) : \deg g \leq d_i\},
    \]
    i.e., polynomials that can be written as a product $f(x_1,\dots,x_k) = f_1(x_1) \cdots f_k(x_k)$. The full result then follows by linearity.
\end{proof}

The second lemma will be used for re-indexing sums in the proof of \Cref{thm:rational}.

\begin{lemma}
\label{lem:rbound}
    Let $r = \lceil \ell/m \rceil$ and $k \geq 0$.
    \begin{enumerate}
        \item[(i)] For $0 \leq d \leq k\ell $, we have $\lfloor d/(km + 1) \rfloor \leq r - 1$,
        \item[(ii)] For $0 \leq d \leq k\ell - n$ and $km \neq 1$, we have $\lfloor d/(km-1) \rfloor \leq r - 1$.
    \end{enumerate}
\end{lemma}
\begin{proof}
This is immediate for $k = 0$, so assume $k \geq 1$. For part (i), observe that
\[
    \frac{d}{km + 1} \leq \frac{k \ell}{km + 1} < \frac{k \ell}{km} = \ell/m \leq r,
\]
and the floor of a number strictly less than $r$ is at most $r-1$. Similarly, for part (ii), we have
\[
    \frac{d}{km - 1} \leq \frac{k\ell - n}{km-1} = \frac{k\ell m - nm}{km^2 - m} < \frac{k\ell m - \ell}{km^2 - m} = \ell/m \leq r.
\]
\end{proof}

The third lemma is the identity responsible for the improved exponent $n-2$ in the case $2\ell = nm$.

\begin{lemma}
\label{lem:2l=nm}
    Let $n \geq 3$, $2\ell = nm$, and $r = \lceil \ell/m \rceil = \lceil n/2 \rceil$. Then
    \[
        \sum_{j=0}^{r-1}(-1)^j \binom{n}{j} (\ell-jm)^{n-2} = 0.
    \]
\end{lemma}
\begin{proof}
    For $0 \leq j \leq n$, let $T(j) = (-1)^j \binom{n}{j} (\ell-jm)^{n-2}$. Since $(\ell-jm)^{n-2}$ has degree $n-2 < n$ as a polynomial in $j$, it follows that $\sum_{j=0}^n T(j) = 0$ (using finite differences). Since $2\ell = nm$, we have $\ell - (n-j)m = -(\ell - jm)$, so
    \[
        T(n-j) = (-1)^{n-j} \binom{n}{j} (-1)^{n-2} (\ell-jm)^{n-2} = T(j).
    \]
    So it suffices to show that $\sum_{r}^{n-r} T(j) = 0$. When $n$ is odd, the sum is empty and thus trivially zero. When $n$ is even, $T(n/2) = 0$ because $(\ell - (n/2)m) = 0$.
\end{proof}

With these lemmas in hand, we are ready to prove \Cref{thm:rational}.

\thmrational*
\begin{proof}
    Using inclusion-exclusion,
    \[
        c(d,km,n) = \sum_{j = 0}^{\lfloor d/(km+1) \rfloor} (-1)^j \binom{n}{j} \binom{d - j(km+1) + n-1}{n-1}.
    \]
    So by Pascal's identity,
    \[
        \delta(d,km,n) = \sum_{j = 0}^{\lfloor d/(km+1) \rfloor} (-1)^j \binom{n}{j} \binom{d - j(km+1) + n-2}{n-2}.
    \]
    Hence, using \Cref{lem:rbound},
    \begin{align*}
        E_{\tope_{\ell,m,n}}(t,q) &= \sum_{k \geq 0} \sum_{d = 0}^{k\ell} \sum_{j=0}^{\lfloor d/(km+1) \rfloor} (-1)^j \binom{n}{j} \binom{d - j(km+1) + n-2}{n-2} q^d t^k \\
        &= \sum_{j = 0}^{r-1} (-1)^j \binom{n}{j} \sum_{k \geq 0} \sum_{e = 0}^{(\ell - jm)k -j} \binom{e + n-2}{n - 2} q^{e + j(km+1)}t^k.
    \end{align*}
    Define $G_\ell$ by extending the inner sum over all $e \geq 0$:
    \begin{align*}
        G_{\ell}(t,q) &= \sum_{j = 0}^{r -1} (-1)^j \binom{n}{j} \sum_{k \geq 0} \sum_{e \geq 0} \binom{e + n-2}{n - 2} q^{e + j(km+1)}t^k \\
        &= \frac{1}{(1-q)^{n-1}} \sum_{j = 0}^{r-1} (-1)^j \binom{n}{j} \frac{q^j}{1-q^{jm} t},
    \end{align*}
    which is rational with denominator dividing $(1-q)^{n-1}\prod_{j=0}^{r - 1} (1-q^{jm} t)$ by \Cref{lem:polyrat}.

    Now consider
    \begin{align*}
        H_\ell(t,q) := G_\ell(t,q) - E_{\tope_{\ell,m,n}}(t,q) &= \sum_{j = 0}^{r -1} (-1)^j \binom{n}{j} \sum_{k \geq 0} \sum_{e \geq (\ell - jm)k - j + 1} \binom{e + n-2}{n - 2} q^{e + j(km+1)}t^k \\
        &= \sum_{j = 0}^{r -1} (-1)^j \binom{n}{j} \sum_{k \geq 0} \sum_{f \geq 0} \binom{f + (\ell - jm)k - j + n - 1}{n - 2} q^{f + \ell k + 1}t^k ,
    \end{align*}
    which is rational with denominator dividing $(1-q)^{n-1}(1-q^\ell t)^{n-1}$ by \Cref{lem:polyrat}.
    
    Hence $E_{\tope_{\ell,m,n}}(t,q)$ is rational with denominator dividing 
    \[
        (1-q)^{n-1}(1-q^\ell t)^{n-1} \prod_{j=0}^{r - 1} (1-q^{jm} t).
    \]
    If we write $E_{\tope_{\ell,m,n}}(t,q) = N/(1-q)^{n-1}D$, where $D = (1-q^\ell t)^{n-1} \prod_{j=0}^{r - 1} (1-q^{jm} t)$, then $N/(1-q)^{n-1} = D\cdot E_{\tope_{\ell,m,n}}(t,q)$. The coefficient of $t^k$ in $D \cdot E_{\tope_{\ell,m,n}}(t,q)$ is a polynomial in $\mathbb{Z}[q]$, so we must have that $(1-q)^{n-1}$ divides the coefficient of $t^k$ in $N$ for all $k$. Hence $(1-q)^{n-1}$ divides $N$. It follows that we can write $E_{\tope_{\ell,m,n}}(t,q)$ as a rational function with denominator dividing $(1-q^\ell t)^{n-1} \prod_{j=0}^{r - 1} (1-q^{jm} t)$.
    
    Now, the coefficient of $k^{n-2}$ in
    \[
        \sum_{j = 0}^{r -1} (-1)^j \binom{n}{j}\sum_{f \geq 0} \binom{f + (\ell - jm)k - j + n-1}{n - 2} q^{f + 1}
    \]
    is $\frac{q}{1-q} \cdot \frac{1}{(n-2)!}$ times
    \[
        \sum_{j = 0}^{r -1} (-1)^j \binom{n}{j}(\ell-jm)^{n-2}.
    \]
    When $2\ell = nm$, this sum is 0 by \Cref{lem:2l=nm}. Hence we get the improved denominator in this case.

    For reciprocity, we have
    \begin{align*}
        G_{\ell}(1/t,1/q) &= \frac{(-1)^n q^{n-1}}{(1-q)^{n-1}}\sum_{j = 0}^{r - 1} (-1)^j \binom{n}{j} \frac{q^{j(m-1)}t}{1 - q^{jm} t} \\
        &= (-1)^n q^{n-1} \sum_{j = 0}^{r - 1}(-1)^j \binom{n}{j} \sum_{k \geq 0} \sum_{e \geq 0} \binom{e + n-2}{n-2}q^{e + jmk + j(m-1)}t^{k + 1} \\
        H_\ell(1/t,1/q) &= \sum_{j = 0}^{r - 1} (-1)^j \binom{n}{j} \sum_{k\geq 1} \sum_{f \geq 1} \binom{-f - (\ell - jm)k - j + n-1}{n-2} q^{f + \ell k -1} t^k \\
        &= \sum_{j = 0}^{r - 1} (-1)^j \binom{n}{j} \sum_{k \geq 0} \sum_{f \geq 0} \binom{-f - (\ell - jm)k - \ell + j(m-1) + n-2}{n-2} q^{f + \ell k + \ell} t^{k+1} \\
        &= (-1)^n \sum_{j = 0}^{r - 1} (-1)^j \binom{n}{j} \sum_{k \geq 0} \sum_{f \geq 0} \binom{f + (\ell - jm)k + \ell - j(m-1) - 1}{n-2} q^{f + \ell k + \ell} t^{k+1} \\
        &= (-1)^n q^{n-1} \sum_{j = 0}^{r - 1} (-1)^j \binom{n}{j} \sum_{k \geq 0} \sum_{e \geq (\ell - jm)k + \ell - j(m-1) - n + 1} \binom{e + n - 2}{n-2} q^{e + jmk +j(m-1)} t^{k+1} \\
    \end{align*}
    by \Cref{lem:polyrat}. By \Cref{cor:intEhrTope} and \Cref{lem:rbound} (when $km = 1$, the inner sum is empty),
    \begin{align*}
        \overline{E}_{\tope_{\ell,m,n}}(t,q) &= \sum_{k \geq 0}\sum_{d = 0}^{k\ell -n} \delta(d,km-2,n) q^d t^k  \\
        &= \sum_{k\geq 0} \sum_{d = 0}^{k \ell - n} \sum_{j = 0}^{\lfloor d/(km-1) \rfloor} (-1)^j \binom{n}{j} \binom{d - j(km-1) + n-2}{n-2}q^d t^{k} \\
        &= \sum_{j = 0}^{r - 1} (-1)^j \binom{n}{j} \sum_{k \geq 0} \sum_{e = 0}^{(\ell - jm)k + j - n} \binom{e + n-2}{n-2} q^{e + j(km-1)}t^{k} \\
        &= \sum_{j = 0}^{r - 1} (-1)^j \binom{n}{j} \sum_{k \geq 0} \sum_{e = 0}^{(\ell - jm)k + \ell - j(m-1) - n} \binom{e + n-2}{n-2} q^{e + jmk + j(m-1)}t^{k+1} \\
    \end{align*}
    Hence
    \[
        q^{n-1}\overline{E}_{\tope_{\ell,m,n}}(t,q) = (-1)^n(G_\ell(1/t,1/q) - H_\ell(1/t,1/q)) = (-1)^n E_{\tope{\ell,m,n}}(1/t,1/q).
    \]
\end{proof}

\begin{corollary}
    For hypersimplices, we have that for $1 \leq \ell \leq \ell_{\max} = \lfloor n/2 \rfloor$, $E_{\Delta_{\ell,n}}(t,q)$ is rational with denominator
    \[
        (1-q^\ell t)^{n-1} \prod_{j=0}^{\ell - 1} (1-q^j t).
    \]  
    Moreover, if $n \geq 3$ and $2\ell = n$, the exponent $n-1$ on $(1-q^\ell t)$ improves to $n-2$.
\end{corollary}

\begin{remark}
    We note an interesting identity relating $q$-Ehrhart series to classical Ehrhart series found while trying to prove \Cref{thm:rational}:
    \[
        E_{\tope_{\ell,m,n}}(t,q) = E_{\tope_{\ell,m,n}}(q^\ell t) + (1-q)\sum_{k \geq 1}\sum_{\substack{0 \leq d < \ell k \\ \gcd(d,k) = 1}} \left(E_{\tope_{d,mk,n}}(q^d t^k) - 1 \right).
    \]
\end{remark}

\begin{example}
Some small examples of $q$-Ehrhart series of hypersimplices are shown in \Cref{tab:small} and \Cref{tab:smallint}.
 
\begin{table}[h!]\centering\footnotesize\renewcommand{\arraystretch}{1.5}
    \begin{tabular}{c c l}
        \hline
        $n$ & $\ell$ & \multicolumn{1}{c}{$E_{\Delta_{\ell,n}}(t,q)$}\\ \hline\noalign{\vskip 3pt}
        $n$ & $1$ & $\dfrac{1}{(1-t)(1-qt)^{n-1}}$ \\[7pt]
        $4$ & $2$ & $\dfrac{1 + 2qt + q^{2}t^{2}}{(1-t)(1-qt)(1-q^{2}t)^{2}}$ \\[7pt]
        $5$ & $2$ & $\dfrac{1 + 3qt + q^{2}t + q^{2}t^{2} - 2q^{3}t^{2} + q^{4}t^{2} - 4q^{4}t^{3} - q^{5}t^{3}}{(1-t)(1-qt)(1-q^{2}t)^{4}}$ \\[7pt]
        $6$ & $2$ & $\dfrac{1 + 4qt + 4q^{2}t + q^{2}t^{2} + 5q^{4}t^{2} - 10q^{4}t^{3} - 4q^{5}t^{3} - q^{6}t^{4}}{(1-t)(1-qt)(1-q^{2}t)^{5}}$ \\[7pt]
        $7$ & $2$ & $\dfrac{1 + 5qt + 8q^{2}t + q^{2}t^{2} + 5q^{3}t^{2} + 15q^{4}t^{2} - 20q^{4}t^{3} - 9q^{5}t^{3} + q^{6}t^{3} - 6q^{6}t^{4} - q^{7}t^{4}}{(1-t)(1-qt)(1-q^{2}t)^{6}}$ \\[7pt]
        \hline
    \end{tabular}
\caption{$q$-Ehrhart series $E_{\Delta_{\ell,n}}(t,q)$.}
\label{tab:small}
\end{table}

\begin{table}[h!]\centering\footnotesize\renewcommand{\arraystretch}{1.5}
    \begin{tabular}{c c l}
        \hline
        $n$ & $\ell$ & \multicolumn{1}{c}{$\overline{E}_{\Delta_{\ell,n}}(t,q)$}\\ \hline\noalign{\vskip 3pt}
        $n$ & $1$ & $\dfrac{t^{n}}{(1-t)(1-qt)^{n-1}}$ \\[7pt]
        $4$ & $2$ & $\dfrac{t^{2} + 2qt^{3} + q^{2}t^{4}}{(1-t)(1-qt)(1-q^{2}t)^{2}}$ \\[7pt]
        $5$ & $2$ & $\dfrac{t^{3} + 4qt^{3} - qt^{4} + 2q^{2}t^{4} - q^{3}t^{4} - q^{3}t^{5} - 3q^{4}t^{5} - q^{5}t^{6}}{(1-t)(1-qt)(1-q^{2}t)^{4}}$ \\[7pt]
        $6$ & $2$ & $\dfrac{t^{3} + 4qt^{4} + 10q^{2}t^{4} - 5q^{2}t^{5} - q^{4}t^{5} - 4q^{4}t^{6} - 4q^{5}t^{6} - q^{6}t^{7}}{(1-t)(1-qt)(1-q^{2}t)^{5}}$ \\[7pt]
        $7$ & $2$ & $\dfrac{t^{4} + 6qt^{4} - qt^{5} + 9q^{2}t^{5} + 20q^{3}t^{5} - 15q^{3}t^{6} - 5q^{4}t^{6} - q^{5}t^{6} - 8q^{5}t^{7} - 5q^{6}t^{7} - q^{7}t^{8}}{(1-t)(1-qt)(1-q^{2}t)^{6}}$ \\[7pt]
        \hline
    \end{tabular}
\caption{Interior $q$-Ehrhart series $\overline{E}_{\Delta_{\ell,n}}(t,q)$.}
\label{tab:smallint}
\end{table}

For $(\ell,n) = (3,6)$,
\[
E_{\Delta_{3,6}}(t,q)=\frac{\begin{array}{c}1 + (4q + 8q^{2} + q^{3})\,t + (q^{2} + 11q^{3} + 10q^{4} - q^{5} + q^{6})\,t^{2}\\[2pt]+ (-q^{4} + q^{5} - 10q^{6} - 11q^{7} - q^{8})\,t^{3} + (-q^{7} - 8q^{8} - 4q^{9})\,t^{4} - q^{10}\,t^{5}\end{array}}{(1-t)(1-qt)(1-q^{2}t)(1-q^{3}t)^{4}}.
\]
and
\[
\overline{E}_{\Delta_{3,6}}(t,q)=\frac{\begin{array}{c}t^{2} + (4q + 8q^{2} + q^{3})\,t^{3} + (q^{2} + 11q^{3} + 10q^{4} - q^{5} + q^{6})\,t^{4}\\[2pt]+ (-q^{4} + q^{5} - 10q^{6} - 11q^{7} - q^{8})\,t^{5} + (-q^{7} - 8q^{8} - 4q^{9})\,t^{6} - q^{10}\,t^{7}\end{array}}{(1-t)(1-qt)(1-q^{2}t)(1-q^{3}t)^{4}}.
\]

Observe that $\overline{E}_{\Delta_{3,6}}(t,q) = t^2E_{\Delta_{3,6}}(t,q)$. This holds more generally whenever $\ell = n/2$ because in these cases $\operatorname{int}(m\Delta_{\ell,n}) \cap \mathbb{Z}^n \cong \lpts{(m-2)\Delta_{\ell,n}}$.
\end{example}

The denominator has $n$ factors if and only if $r=1$, or $r=2$ and $2\ell = nm$. The case $r=1$ happens when $\ell \leq m$, i.e., when $\tope_{\ell,m,n}$ is a simplex. The case $r=2$ and $2\ell = nm$ forces $n \in \{3,4\}$ and gives two non-simplex families:
\begin{itemize}
    \item dilates of the hexagon $\tope_{3,2,3}$,
    \item dilates of the octahedron $\tope_{2,1,4} = \Delta_{2,4}$.
\end{itemize}
In both cases, computational evidence suggests that the numerator always has nonnegative coefficients as a polynomial in $q,t$. We computed $E_{\tope_{\ell,m,n}}(t,q)$ for all $n \leq 9$, $m \leq 7$, and $1 \leq \ell \leq \ell_{\max}$. In all computed cases, the numerator lies in $\mathbb{N}[t,q]$ if and only if $\nu = n$. This evidence suggests that part (ii) of \Cref{conj:series} holds for a larger class of polytopes than just simplices.

\section{The harmonic algebra}
\label{sec:harm}

Let $V_{m\Delta_{\ell,n}}$ denote the harmonic space of the ideal $\gri{\lpts{m\Delta_{\ell,n}}}$. Let $\mathcal{H}_{\Delta_{\ell,n}}$ denote the harmonic algebra of the hypersimplex $\Delta_{\ell,n}$. Recall from \Cref{sec:back} that this algebra is defined as
\[
    \mathcal{H}_{\Delta_{\ell,n}} = \bigoplus_{m \geq 0} \mathbb{Q} \cdot x_0^m \otimes V_{m\Delta_{\ell,n}}.
\]
In this section we will show that $\mathcal{H}_{\Delta_{\ell,n}}$ is generated as an algebra in $x_0$-degree one by $V_{\Delta_{\ell,n}}$. That is, the algebra map
\[
    \mathbb{Q}[V_{\Delta_{\ell,n}}] \to \mathcal{H}_{\Delta_{\ell,n}} 
\]
determined by $g \mapsto x_0 \cdot g$ for $g \in V_{\Delta_{\ell,n}}$ is a surjection.

We will assume throughout this section that $\ell \leq \ell_{\max} = \lfloor n/2 \rfloor$. We recall that $\Delta_{\ell,n} \cong \Delta_{n-\ell,n}$, so we lose nothing by making this restriction on $\ell$. We denote by $G(\ell,m,n)$ the set of loopless multigraphs on $n$ vertices labeled $\{1,2,\dots,n\}$ with number of edges $\leq \ell m$ and vertex degrees $\leq m$. For $G \in G(\ell,m,n)$, we define
\[
    \delta_G := \prod_{(i,j) \in G}(x_i - x_j)
\]
where $(i,j)$ ranges over all edges of $G$ (with multiplicity) with $i < j$. So $\delta_G$ is homogeneous with degree equal to the number of edges in $G$ counted with multiplicity. An example is given in \Cref{ex:easygraph}. The next lemma shows that $\delta_G \in V_{m \Delta_{\ell,n}}$. We will often tacitly identify the graph $G \in G(\ell,m,n)$ with the polynomial $x_0^m\delta_G \in \mathbb{Q} \cdot x_0 \otimes V_{m\Delta_{\ell,n}} \subset \mathcal{H}_{\Delta_{\ell,n}}$.

\begin{lemma}
    We have that $\delta_G \in V_{m\Delta_{\ell,n}}$ for $G \in G(\ell,m,n)$.
\end{lemma}

\begin{proof}
We can check that for each generator $f$ of $\gri{\lpts{m\Delta_{\ell,n}}}$ we have $f \odot \delta_G = 0$. By \Cref{thm:gri}, there are three types of generators of this ideal:
\begin{enumerate}
    \item $x_i^{m+1}$
    \item $x_1^{a_1}\dots x_n^{a_n}$ for $a_1 + \dots + a_n = \ell m + 1$
    \item $L = x_1 + \dots + x_n$
\end{enumerate}

For type (1), $x_i^{m+1} \odot \delta_G = 0$ because the vertex degrees of $G$ are at most $m$.

For type (2), $x_1^{a_1}\dots x_n^{a_n} \odot \delta_G = 0$ because there are at most $\ell m$ edges in $G$.

For type (3), we use the product rule
\begin{align*}
    L \odot \delta_G  = \sum_{(i,j) \in G} (L \odot  (x_i-x_j)) \cdot \prod_{\substack{(a,b) \in G \\ (a,b) \neq (i,j)}} (x_a - x_b) = 0,
\end{align*}
where the last equality uses that $L \odot (x_i - x_j) = 0$ for all $(i,j)$.
\end{proof}

For an element $G \in G(\ell,m,n)$ with $d$ edges, we define a tableau of shape $2 \times d$ as follows: order the edges lexicographically and fill the tableau so that the first row records the smaller endpoints and the second the larger. The first row $(a_1,\dots,a_d)$ is automatically non-decreasing. If the second row $(b_1,\dots,b_d)$ is also non-decreasing, then we have a two-row semistandard Young tableau $T$. We will let $S(\ell,m,n)$ denote the set of all such semistandard tableaux. The set $S(\ell,m,n)$ can be equivalently defined as all graphs $G \in G(\ell,m,n)$ satisfying a nonnesting property: for all vertices $a < b < c < d$, $G$ cannot have both edges $(a,d)$ and $(b,c)$. 

\begin{example}
\label{ex:easygraph}
    Let $G$ be the graph on $\{1,2,3,4\}$ with edges $(1,3)$, $(2,4)$ and a double edge
    $(3,4)$. Ordering the edges lexicographically gives
    \[
      (1,3),\quad (2,4),\quad (3,4),\quad (3,4),
    \]
    so the first row is $(a_1,\dots,a_4)=(1,2,3,3)$ and the second is
    $(b_1,\dots,b_4)=(3,4,4,4)$. Both rows are non-decreasing, so $G$ gives a
    two-row semistandard Young tableau.
    \[
      T \;=\;
      \begin{tikzpicture}[baseline=-2mm,x=6.2mm,y=6.2mm]
        \foreach \c/\v in {0/1,1/2,2/3,3/3}{\draw (\c,0) rectangle ++(1,1);
          \node at (\c+0.5,0.5) {$\v$};}
        \foreach \c/\v in {0/3,1/4,2/4,3/4}{\draw (\c,-1) rectangle ++(1,1);
          \node at (\c+0.5,-0.5) {$\v$};}
      \end{tikzpicture}
      \qquad\text{for}\qquad
      G \;=\; 
      \begin{tikzpicture}[baseline=-1mm,x=1.05cm,y=1cm,font=\small,
          v/.style={circle,draw,inner sep=0pt,minimum size=6.4mm,fill=white}]
        \foreach \x/\lab in {0/1,1/3,2/4,3/2}{\node[v] (m\x) at (\x,0) {$\lab$};}
        \draw (m0) -- (m1);
        \draw (m1) to[bend left=42] (m2);
        \draw (m1) to[bend right=42] (m2);
        \draw (m2) -- (m3);
      \end{tikzpicture}
    \]
    We have that $G\in S(\ell,m,n)$ for $n = 4$ and any $\ell,m$ with $m\ge3$ and $\ell m\ge4$. The polynomial $\delta_G$ is
    \[
        \delta_G = (x_1 - x_3)(x_2-x_4)(x_3 - x_4)^2.
    \]
\end{example}

\begin{lemma}
\label{lem:span}
    For $m \geq 1$, $V_{m\Delta_{\ell,n}}$ is linearly spanned by $\{\delta_T: T \in S(\ell, m, n)\}$. The dimension of $V_{m\Delta_{\ell,n}}$ is $c(\ell m, m ,n)$ which also equals the number of realizable first rows of elements of $G(\ell,m,n)$.
\end{lemma}
\begin{proof}
    See \cite[Lemma 3.3 and Proposition 3.4]{zigzags}.
\end{proof}

\subsection{Multiplicative structure}

The remainder of this section will primarily use techniques and objects in graph theory, so we begin by recalling some standard definitions. A graph is said to be \defn{$k$-regular} if every vertex has degree $k$. A \defn{$k$-factor} of a graph $G$ is a $k$-regular subgraph containing every vertex of $G$. So if a $k$-regular graph $R$ has an $m$-factor $F$, then the edge complement $R-F$ is a $(k-m)$-factor of $R$.

Let $e_G(i,j)$ denote the number of edges between $i$ and $j$ in a multigraph (the notation is symmetric in $i$ and $j$; we do not require that $i < j$ here). Given $G_1 \in G(\ell,m,n)$ and $G_2 \in G(\ell, m',n)$, we define $G_1\cdot G_2 \in G(\ell, m + m', n)$ to be the multigraph on $\{1,\dots,n\}$ with $e_{G_1\cdot G_2}(i,j) = e_{G_1}(i,j) + e_{G_2}(i,j)$ for all $i,j$. Correspondingly, $x_0^m\delta_{G_1}\cdot x_0^{m'}\delta_{G_2} = x_0^{m+m'}\delta_{G_1 \cdot G_2}$ in $\mathcal{H}_{\Delta_{\ell,n}}$. We write
\[
    G(\ell,m,n) \cdot G(\ell,m',n) = \{G_1\cdot G_2 : G_1 \in G(\ell,m,n) \text{ and } G_2 \in G(\ell,m',n)\} \subseteq G(\ell,m + m', n).
\]
In this section, we will prove that the above inclusion is always an equality when $m$ and $m'$ are both even. The other cases will be investigated in \Cref{app:modd}. As a first result in that direction, we have the following lemma.

\begin{lemma}
    If $G \in G(\ell,m,n) \cdot G(\ell,m',n)$, then $G$ is a subgraph of an $(m + m')$-regular graph with an $m$-factor.
\end{lemma}
\begin{proof}
    Take two disjoint copies of $G$ and connect each vertex $v \in G$ with its copy using $m + m' - \deg(v)$ edges. This produces an $(m + m')$-regular graph $R$. Write $G = G_1 \cdot G_2$ with $G_1 \in G(\ell,m,n)$. For a vertex $v \in G$, let $e_{G_1}(v)$ denote the set of edges incident to $v$ in $G_1$. Then $R$ has an $m$-factor consisting of the edges from $G_1$ (in both copies of $G$) and $m - \#e_{G_1}(v)$ of the edges between the copies of $v$.
\end{proof}

The other direction of this lemma is false, but offers inspiration for our approach. First, we establish some terminology. For $G \in G(\ell,m,n)$, let $B_G$ be the set of cut-edges in $G$. We refer to the connected components of $G-B_G$ as the \defn{blocks} of $G$. We define the \defn{deficit} of a vertex by $\operatorname{def}(v):= m - \deg(v)$. The deficit of a block $B$ is
\[
    \operatorname{def}(B) = \sum_{v \in B} \operatorname{def}(v).
\]
The \defn{total deficit} of a connected component is the sum of the deficits of all its blocks. We will denote by $C_d(G)$ the set of components of $G$ with total deficit $d$. We will say a block is \defn{full} if it has deficit 0 and is \defn{almost full} if it has deficit 1. 

Let $T_G$ denote the \defn{bridge forest} of $G$ with vertices as the blocks of $G$ and edges as the edges in $B_G$. A block is a \defn{leaf-block} if it has degree 1 in $T_G$ (a block with degree 0 in $T_G$ is not considered a leaf-block). We will denote by $FL(G)$ and $AFL(G)$ the sets of full leaf-blocks and almost full leaf-blocks of $T_G$, respectively.

\begin{example}\label{ex:blocks}
Let $m = 3$, $\ell = 4$ and $n = 8$, and let $G \in G(\ell,m,n)$ have the ten edges
\[
  \underbrace{(1,2),\; (1,2),\; (1,3),\; (2,3)}_{\text{block } B_1},\quad
  (3,4),\quad
  \underbrace{(4,5),\; (4,5)}_{\text{block } B_2},\quad
  (5,6),\quad
  \underbrace{(6,7),\; (6,7)}_{\text{block } B_3},
\]
with the vertex $8$ isolated (\Cref{fig:blocks}). Every vertex degree is at most $3$ and $\ell m = 12 \geq 10$, so indeed $G \in G(\ell,m,n)$. 

The cut-edges are
\[
  B_G = \{(3,4),\; (5,6)\}
\]
and the blocks of $G$ are
\[
  B_1 = \{1,2,3\}, \qquad B_2 = \{4,5\}, \qquad B_3 = \{6,7\}, \qquad B_4 = \{8\}.
\]
Since $\operatorname{def}(v) = 3 - \deg(v)$, the only vertices of positive deficit are $7$ and $8$, with $\operatorname{def}(7) = 1$ and $\operatorname{def}(8) = 3$. Hence
\[
  \operatorname{def}(B_1) = \operatorname{def}(B_2) = 0, \qquad
  \operatorname{def}(B_3) = 1, \qquad
  \operatorname{def}(B_4) = 3,
\]
so $B_1$ and $B_2$ are full and $B_3$ is almost full. The graph $G$ has two connected components: the one containing $1, \dots, 7$ has total deficit $1$, while $\{8\}$ has total deficit $3$. Thus $C_1(G)$ consists of the large component, $C_3(G) = \bigl\{\{8\}\bigr\}$, and $C_d(G) = \varnothing$ for all other $d$.

The bridge forest $T_G$ has vertices $B_1, B_2, B_3, B_4$, with $(3,4)$ joining $B_1$ to $B_2$ and $(5,6)$ joining $B_2$ to $B_3$. So $T_G$ is the disjoint union of the path $B_1 - B_2 - B_3$ and the isolated vertex $B_4$. Since $B_1$ and $B_3$ have degree $1$ in $T_G$ while $B_2$ has degree $2$ and $B_4$
has degree $0$,
\[
  FL(G) = \{B_1\}, \qquad AFL(G) = \{B_3\}.
\]
\end{example}

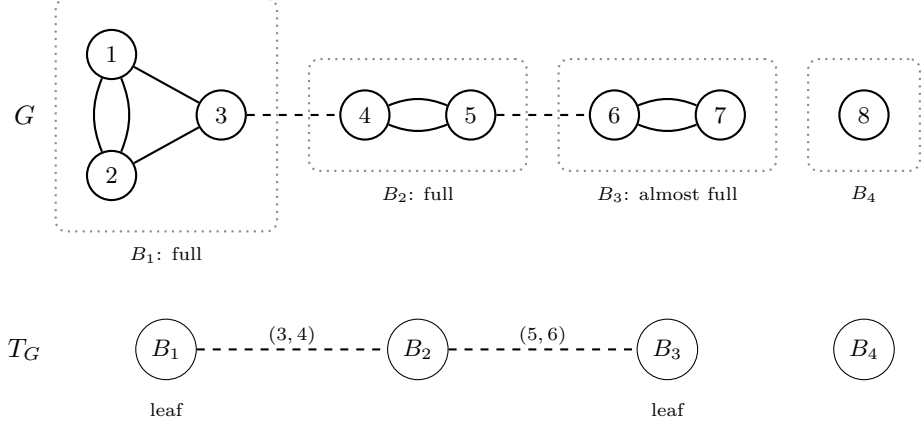
\begin{figure}[ht]
\centering
\begin{tikzpicture}[
    vtx/.style={circle, draw, thick, inner sep=0pt, minimum size=6.5mm, font=\small},
    def/.style={vtx, fill=black!12},
    blk/.style={circle, draw, inner sep=0pt, minimum size=8mm, font=\small},
    cut/.style={thick, dashed},
    box/.style={rounded corners, draw=black!45, dotted, thick, inner sep=4mm},
    elbl/.style={font=\scriptsize, inner sep=1.5pt}
  ]

  \node[vtx] (v1) at (0,0.8)   {$1$};
  \node[vtx] (v2) at (0,-0.8)  {$2$};
  \node[vtx] (v3) at (1.45,0)  {$3$};
  \node[vtx] (v4) at (3.35,0)  {$4$};
  \node[vtx] (v5) at (4.75,0)  {$5$};
  \node[vtx] (v6) at (6.65,0)  {$6$};
  \node[vtx] (v7) at (8.05,0)  {$7$};
  \node[vtx] (v8) at (9.95,0)  {$8$};

  \draw[thick] (v1) to[bend left=22]  (v2);
  \draw[thick] (v1) to[bend right=22] (v2);
  \draw[thick] (v1) -- (v3);
  \draw[thick] (v2) -- (v3);

  \draw[cut] (v3) -- node[elbl, above] {} (v4);

  \draw[thick] (v4) to[bend left=22]  (v5);
  \draw[thick] (v4) to[bend right=22] (v5);

  \draw[cut] (v5) -- node[elbl, above] {} (v6);

  \draw[thick] (v6) to[bend left=22]  (v7);
  \draw[thick] (v6) to[bend right=22] (v7);

  \begin{scope}[on background layer]
    \node[box, fit=(v1)(v2)(v3)] {};
    \node[box, fit=(v4)(v5)]     {};
    \node[box, fit=(v6)(v7)]     {};
    \node[box, fit=(v8)]         {};
  \end{scope}

  \node[font=\scriptsize] at (0.72,-1.85) {$B_1$: full};
  \node[font=\scriptsize] at (4.05,-1.05) {$B_2$: full};
  \node[font=\scriptsize] at (7.35,-1.05) {$B_3$: almost full};
  \node[font=\scriptsize] at (9.95,-1.05) {$B_4$};

  \node[font=\itshape] at (-1.15,0) {$G$};

  \begin{scope}[yshift=-3.1cm]
    \node[blk] (b1) at (0.72,0) {$B_1$};
    \node[blk] (b2) at (4.05,0) {$B_2$};
    \node[blk] (b3) at (7.35,0) {$B_3$};
    \node[blk] (b4) at (9.95,0) {$B_4$};

    \draw[cut] (b1) -- node[elbl, above] {$(3,4)$} (b2);
    \draw[cut] (b2) -- node[elbl, above] {$(5,6)$} (b3);

    \node[font=\scriptsize] at (0.72,-0.8) {leaf};
    \node[font=\scriptsize] at (4.05,-0.8) {};
    \node[font=\scriptsize] at (7.35,-0.8) {leaf};
    \node[font=\scriptsize] at (9.95,-0.8) {};

    \node[font=\itshape] at (-1.15,0) {$T_G$};
  \end{scope}

\end{tikzpicture}
\caption{The graph $G$ of \Cref{ex:blocks} and its bridge forest $T_G$. The two cut-edges are dashed and the four blocks are enclosed by dotted boxes. The blocks $B_1$ and $B_3$ are the leaves of $T_G$, so $FL(G) = \{B_1\}$ and $AFL(G) = \{B_3\}$. The block $B_2$ is full but has degree $2$, and $B_4$ has degree $0$, so neither is a leaf-block.}
\label{fig:blocks}
\end{figure}

\begin{theorem}
\label{thm:2f}
    Suppose that $R$ is an $m$-regular loopless multigraph. Then $R$ has a 2-factor if and only if one of the following conditions is true:
    \begin{enumerate}
        \item $m$ is even \cite{peterson}, or
        \item $m\geq 3$ is odd and every 2-factor-free block of $R$ is incident to fewer than $m$ cut-edges \cite{minimal}.
    \end{enumerate}
\end{theorem}

\begin{corollary}
\label{cor:2fblock}
    If $R$ is an $m$-regular loopless multigraph with $m \geq 3$ odd, and each connected component of $T_R$ has fewer than $m$ leaf-blocks, then $R$ has a 2-factor.
\end{corollary}
\begin{proof}
    If there is a block incident to at least $m$ cut-edges, then $T_R$ must have at least $m$ leaf-blocks in the same connected component.
\end{proof}

We will use \Cref{thm:2f} by taking an arbitrary element of $G(\ell,m,n)$, embedding it in an $m$-regular graph with a 2-factor, then restricting the 2-factor back to the original graph to show that it is in $G(\ell,2,n) \cdot G(\ell,m-2,n)$. But we need to account for the $\ell$ parameter when we restrict back to the original graph by controlling how many edges of the 2-factor are contained in the original graph.

We will say $R$ is an \defn{external regularization} of $G$ if it can be constructed using the following steps:
\begin{enumerate}
    \item If $G$ is regular, then the only external regularization of $G$ is $R = G$ itself. Otherwise, define $k := \lfloor 2d/m \rfloor \leq 2\ell$, where $d$ is the number of edges in $G$. Then add $n - k - 1$ vertices, labeled $v_1,v_2,\dots,v_{n-k-1}$. Connect each of these vertices with $m$ edges to vertices in $G$ such that the maximum degree in $G$ is still $\leq m$ (which is possible by choice of $k$). After this step, the sum of the vertex degrees in $G$ is 
    \[
        mn - (m(k+1) -2d).
    \]

    \item Add an additional vertex $v_{n-k}$ and connect it with vertices in $G$ using $m(k+1) -2d \leq m$ edges, again keeping vertex degrees $\leq m$. After this step, every vertex in $G$ has degree $m$. We will refer to the $v_1,\dots,v_{n-k}$ as the \defn{external vertices}.

    \item If $2d/m$ is an integer, we now have an $m$-regular graph. Otherwise, if $mk$ is even, we add edges and vertices as shown on the left side of \Cref{fig:cases}. If $mk$ is odd, then $m$ must also be odd. In this case we add edges and vertices as shown on the right side of \Cref{fig:cases}.

    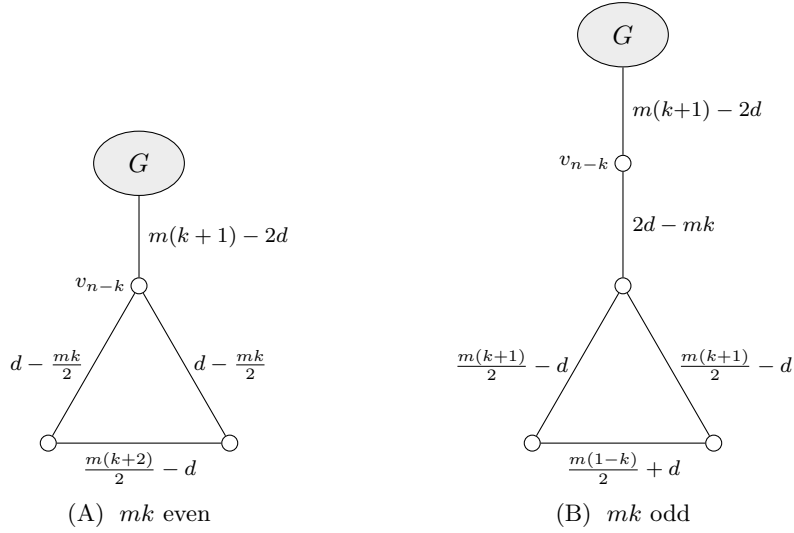
\begin{figure}[ht]
      \centering
      \begin{tikzpicture}[
          vtx/.style  = {circle, draw, minimum size=6pt, inner sep=0pt},
          gbox/.style = {draw, ellipse, fill=black!7, minimum width=12mm, minimum height=8.5mm, inner sep=1pt},
          lbl/.style  = {font=\footnotesize},
          nm/.style   = {font=\footnotesize, inner sep=1.5pt},
          sub/.style  = {font=\small},
        ]
        \begin{scope}[shift={(0,0)}]
          \node[gbox] (G) at (0,-1.4)      {$G$};
          \node[vtx, label={[nm]left:$v_{n-k}$}] (v) at (0,-3.02) {};
          \node[vtx]  (l) at (-1.2,-5.1)   {};
          \node[vtx]  (r) at (1.2,-5.1)    {};
          \draw (G) -- (v) node[lbl, right, midway]  {$m(k+1)-2d$};
          \draw (v) -- (l) node[lbl, left, midway]   {$d-\tfrac{mk}{2}$};
          \draw (v) -- (r) node[lbl, right, midway]  {$d-\tfrac{mk}{2}$};
          \draw (l) -- (r) node[lbl, below, midway]  {$\frac{m(k+2)}{2}-d$};
          \node[sub] at (0,-6.05) {(A)\; $mk$ even};
        \end{scope}
        \begin{scope}[shift={(6.4,0)}]
          \node[gbox] (G2) at (0,0.3)      {$G$};
          \node[vtx, label={[nm]left:$v_{n-k}$}] (v2) at (0,-1.4) {};
          \node[vtx]  (w2) at (0,-3.02)    {};
          \node[vtx]  (l2) at (-1.2,-5.1)  {};
          \node[vtx]  (r2) at (1.2,-5.1)   {};
          \draw (G2) -- (v2) node[lbl, right, midway]  {$m(k{+}1)-2d$};
          \draw (v2) -- (w2) node[lbl, right, midway]  {$2d-mk$};
          \draw (w2) -- (l2) node[lbl, left, midway]   {$\tfrac{m(k+1)}{2}-d$};
          \draw (w2) -- (r2) node[lbl, right, midway]  {$\tfrac{m(k+1)}{2}-d$};
          \draw (l2) -- (r2) node[lbl, below, midway]  {$\tfrac{m(1-k)}{2}+d$};
          \node[sub] at (0,-6.05) {(B)\; $mk$ odd};
        \end{scope}
      \end{tikzpicture}
      \caption{Step (3) of the external regularization, in the case $2d/m \notin \mathbb{Z}$. The new vertices are joined only to $v_{n-k}$ and to each other. All vertex degrees come out to exactly $m$, so the result is $m$-regular.}
      \label{fig:cases}
    \end{figure}
\end{enumerate}

Since none of the added vertices $v_1,\dots,v_{n-k}$ have edges between them, if $R$ has a 2-factor, then we know how many of its edges are in the original graph. This gives us the following result.

\begin{theorem}
\label{thm:extreg}
    If $G \in G(\ell,m,n)$ has an external regularization with a 2-factor, then
    \[
        G \in G(\ell,2,n) \cdot G(\ell,m-2,n).
    \]
\end{theorem}
\begin{proof}
    Let $F$ be the 2-factor. There are two cases:
    \begin{enumerate}
        \item[Case 1.] If $2d/m$ is an integer, then $v_{n-k}$ connects only to vertices in $G$. So the number of edges of $F$ not in $G$ is exactly $2n - 2k$. Hence the restriction $F|_G$ of $F$ to $G$ has exactly $k \leq 2\ell$ edges, so $F|_G \in G(\ell,2,n)$. Its edge-complement $F|_G^c$ has vertex degrees $\leq m-2$ and has exactly $(m-2)d/m \leq (m-2)\ell$ edges, so $F|_G^c \in G(\ell, m-2,n)$.

        \item[Case 2.] If $2d/m$ is not an integer, then by a similar analysis, the number of edges in the restriction $F|_G$ is either $k$ or $k + 1$. Also, $k \leq 2\ell - 1$, so $F|_G \in G(\ell,2,n)$. Now,
        \[
            d-(k+1) \leq d-k = \lceil (m-2)d/m\rceil \leq (m-2)\ell.
        \]
        Hence $F|_G^c \in G(\ell, m-2,n)$.
    \end{enumerate}

    Thus $G = F|_G \cdot F|_G^c$ with $F|_G \in G(\ell, 2,n)$ and $F|_G^c \in G(\ell, m-2,n)$ as desired.
\end{proof}

\begin{corollary}
\label{cor:even}
    If $m$ is even, then
    \[
        G(\ell,m,n) = G(\ell,2,n) \cdot G(\ell,m-2,n).
    \]
\end{corollary}
\begin{proof}
    Take any external regularization $R$ of $G \in G(\ell,m,n)$. By \Cref{thm:2f}, $R$ has a 2-factor. Hence $G \in G(\ell,2,n) \cdot G(\ell,m-2,n)$ by \Cref{thm:extreg}.
\end{proof}

The case when $m$ is odd is more complicated and is handled in \Cref{app:modd}.

\subsection{Maximal tableaux}

For any possible first row of an element of $G(\ell,m,n)$ there is a unique maximal non-decreasing second row obtained by choosing $b_d$ as large as possible, then $b_{d-1}$ as large as possible, and so on. We will denote by $S_{\max}(\ell,m,n)$ the set of graphs with maximal non-decreasing second row. Because the second row is non-decreasing, maximal tableaux are semistandard, i.e., elements of $S(\ell,m,n)$.

\begin{lemma}
\label{lem:basis}
    $S_{\max}(\ell,m,n)$ is a basis of $V_{m\Delta_{\ell,n}}$.
\end{lemma}
\begin{proof}
    The elements of $S_{\max}(\ell,m,n)$ are linearly independent because they have distinct leading monomials in lexicographic order. By \Cref{lem:span}, $S_{\max}(\ell,m,n)$ has size $\dim V_{m\Delta_{\ell,n}}$, so is a basis.
\end{proof}

Denote by $S_{\reg}(m,n)$ the set of $m$-regular graphs in $S(\lceil n/2 \rceil,m,n)$. The tableaux are semistandard since they are in $S(\lceil n/2 \rceil,m,n)$, and $m$-regularity determines the second row from the first, so the second row is certainly maximal. Hence $S_{\reg}(m,n) \subseteq S_{\max}(\lceil n/2 \rceil,m,n)$.

\begin{lemma}
\label{lem:maxreg}
    Every element of $S_{\max}(\ell,m,n)$ has an external regularization in $S_{\reg} (m,n')$ for some $n' \geq n$.
\end{lemma}
\begin{proof}
    Extend the tableau $T \in S_{\max}(\ell,m,n)$ to the left by adding $\operatorname{def}(v)$ many copies of $v$ to the second row, then completing the first row with nonpositive integers and capping the tableau on the left according to \Cref{fig:cases}. Consider the example in \Cref{fig:tab}.
\end{proof}

\begin{figure}[h]
\definecolor{f8gad}{RGB}{250,216,168}   
\definecolor{f8stp2}{RGB}{198,227,190}  
\definecolor{f8stp1}{RGB}{198,216,240}  
\tikzset{
  f8cell/.style = {rectangle, draw, minimum width=0.5cm, minimum height=0.5cm,
                   inner sep=1pt, font=\scriptsize},
  f8vtx/.style  = {circle, draw, thick, inner sep=0pt, minimum size=6mm,
                   font=\small},
  f8e/.style    = {thick},
  f8cut/.style  = {thick, densely dashed},
  f8box/.style  = {draw=black!55, densely dotted, thick, rounded corners=4pt},
  f8arr/.style  = {-{Stealth[length=5pt]}, thick, draw=black!55},
  f8lab/.style  = {font=\footnotesize\itshape, text=black!70},
}
\centering
\begin{tikzpicture}
 
\node[f8lab, anchor=south] at (1.54,0.32) {$T$};
\foreach \a/\b [count=\i from 0] in {1/3, 2/4, 3/4, 3/4}{
  \node[f8cell, fill=white] at (0.79+\i*0.5, 0.00) {$\a$};
  \node[f8cell, fill=white] at (0.79+\i*0.5,-0.50) {$\b$};
}
 
\draw[f8arr] (2.94,-0.25) -- (3.64,-0.25);
 
\node[f8lab, anchor=south] at (5.325,0.32) {$G$};
\node[f8vtx, fill=white] (g1) at (4.05,-0.25) {$1$};
\node[f8vtx, fill=white] (g3) at (4.90,-0.25) {$3$};
\node[f8vtx, fill=white] (g4) at (5.75,-0.25) {$4$};
\node[f8vtx, fill=white] (g2) at (6.60,-0.25) {$2$};
\draw[f8e] (g1) -- (g3);
\draw[f8e] (g3) to[bend left=30]  (g4);
\draw[f8e] (g3) to[bend right=30] (g4);
\draw[f8e] (g4) -- (g2);
 
\node[f8lab, anchor=south] at (3.80,-1.34)
      {completion of $T$ by $0$ and negative entries};
\foreach \a/\b/\c [count=\i from 0] in {%
    -3/-2/f8gad, -3/-2/f8gad, -3/-1/f8gad, -2/-1/f8gad,
    -1/1/f8stp2,
    0/1/f8stp1, 0/2/f8stp1, 0/2/f8stp1,
    1/3/white, 2/4/white, 3/4/white, 3/4/white}{
  \node[f8cell, fill=\c] at (1.05+\i*0.5,-1.70) {$\a$};
  \node[f8cell, fill=\c] at (1.05+\i*0.5,-2.20) {$\b$};
}
 
\draw[f8arr] (3.80,-2.50) -- (3.80,-2.95);
\node[f8lab, anchor=west] at (3.95,-2.725) {re-index};
 
\foreach \a/\b/\c [count=\i from 0] in {%
    1/2/f8gad, 1/2/f8gad, 1/3/f8gad, 2/3/f8gad,
    3/5/f8stp2,
    4/5/f8stp1, 4/6/f8stp1, 4/6/f8stp1,
    5/7/white, 6/8/white, 7/8/white, 7/8/white}{
  \node[f8cell, fill=\c] at (1.05+\i*0.5,-3.25) {$\a$};
  \node[f8cell, fill=\c] at (1.05+\i*0.5,-3.75) {$\b$};
}
 
\node[f8lab, anchor=south] at (3.80,-4.55) {$R$};
 
\draw[f8box] (3.00,-6.48) rectangle (7.50,-5.52);
\node[f8lab, anchor=south, font=\scriptsize\itshape] at (5.25,-5.52) {$G$};
 
\node[f8vtx, fill=f8gad]  (r1) at (0.55,-5.05) {$1$};
\node[f8vtx, fill=f8gad]  (r2) at (0.55,-6.95) {$2$};
\node[f8vtx, fill=f8stp2] (r3) at (2.25,-6.00) {$3$};
\node[f8vtx, fill=f8stp1] (r4) at (5.25,-8.20) {$4$};
\node[f8vtx, fill=white]  (r5) at (3.45,-6.00) {$5$};
\node[f8vtx, fill=white]  (r7) at (4.65,-6.00) {$7$};
\node[f8vtx, fill=white]  (r8) at (5.85,-6.00) {$8$};
\node[f8vtx, fill=white]  (r6) at (7.05,-6.00) {$6$};
 
\draw[f8e]   (r1) to[bend left=28]  (r2);
\draw[f8e]   (r1) to[bend right=28] (r2);
\draw[f8e]   (r1) -- (r3);
\draw[f8e]   (r2) -- (r3);
\draw[f8cut] (r3) -- (r5);
\draw[f8e]   (r4) -- (r5);
\draw[f8e]   (r4) to[bend left=18]  (r6);
\draw[f8e]   (r4) to[bend right=18] (r6);
\draw[f8e]   (r5) -- (r7);
\draw[f8e]   (r7) to[bend left=28]  (r8);
\draw[f8e]   (r7) to[bend right=28] (r8);
\draw[f8e]   (r8) -- (r6);
 
\foreach \x/\y/\c/\t in {%
    0.40/-8.90/f8gad/{\Cref{fig:cases} construction},
    4.00/-8.90/f8stp2/{$v_{n-k}$ (step 2)},
    0.40/-9.35/f8stp1/{step-1 vertex},
    4.00/-9.35/white/{vertices of $G$}}{
  \node[f8cell, fill=\c, minimum width=0.30cm, minimum height=0.30cm]
        at (\x,\y) {};
  \node[f8lab, anchor=west] at (\x+0.24,\y) {\t};
}
 
\end{tikzpicture}
 
\caption{The tableau completion of \Cref{lem:maxreg} and the graphs it encodes, for $T$ with rows $1\,2\,3\,3$ and $3\,4\,4\,4$. Reading the completed tableau left to right reads off the external regularization: columns $1$--$4$ are the \Cref{fig:cases} construction, column $5$ is the single edge from $v_{n-k}$ into $G$, columns $6$--$8$ are the $m=3$ edges from the step-1 external vertex into $G$, and columns $9$--$12$ are the edges of $G$ itself. The only cut-edge of $R$ is $(3,5)$, so the bridge forest of $R$ is a single path, as required by \Cref{cor:path}.}
\label{fig:tab}
\end{figure}
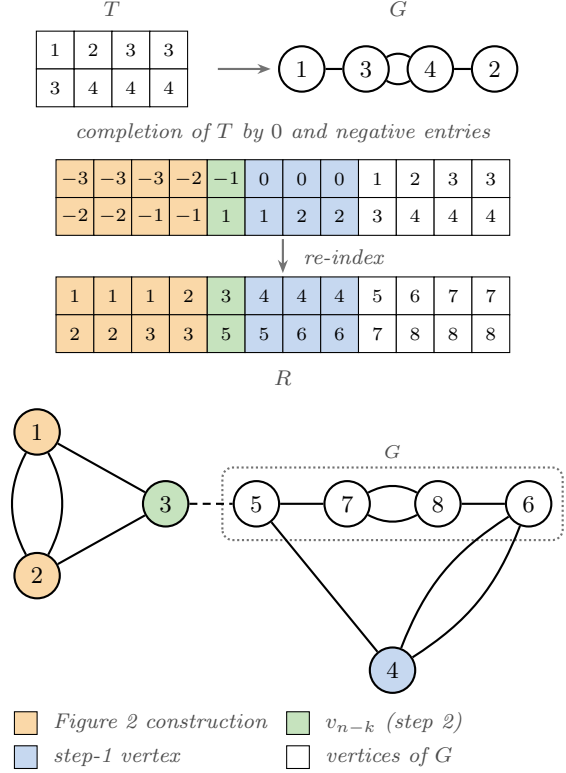

\begin{lemma}
\label{lem:bipartite}
    Let $G \in S(\ell,m,n)$ be connected, let $f$ be a cut-edge of $G$, and let $Y$ and $Z$ be the components of $G - f$. If $Z$ contains the vertices labeled $1$ and $n$, then $Y$ must be bipartite.
\end{lemma}
\begin{proof}
    Fix a path $1 = z_0, z_1, \dots, z_N = n$ from $1$ to $n$ in $Z$ with edges $e_1,\dots, e_N$, where $e_t$ is the edge between $z_{t-1}$ and $z_t$. Define a function on vertices of $Y$ by
    \[
        \nu(y) := \#\{t : \text{both ends of $e_t$ are less than $y$}\}.
    \]
    Color $y$ by the parity of $\nu(y)$. In \Cref{ex:big}, we will use red for $\nu(y)$ odd and blue for $\nu(y)$ even. To show that $Y$ is bipartite, we need to check that every edge in $Y$ goes between vertices of opposite color.

    Suppose that there is an edge between $y$ and $y'$, with $y < y'$. Split the other vertices into sets
    \[
        L = \{x : x < y\}, \qquad M = \{x : y < x < y'\}, \qquad R = \{x : x > y'\}.
    \]
    From the fact that $G \in S(\ell,m,n)$, we know that
    \begin{enumerate}
        \item[(i)] no $e_t$ has both ends in $M$;
        \item[(ii)] no $e_t$ has one end in $L$ and the other in $R$.
    \end{enumerate}
    See \Cref{fig:LMR}

    It suffices to show that $\nu(y') - \nu(y)$ is odd. An edge counted by $\nu(y')$ but not by $\nu(y)$ is one whose large end lies in $M$. We know that the smaller end cannot also be in $M$ by (i), so it must lie in $L$. Hence
    \[
        \nu(y') - \nu(y) = \#\{t : e_t \text{ joins $L$ to $M$}\}.
    \]
    Read off the sequence of regions $L,M,R$ containing $z_0,z_1,\dots,z_N$. By (i) no two consecutive entries are $M$, so each visit to $M$ is a single vertex $z_t$ with $z_{t-1},z_{t+1} \in L \cup R$. We have that $z_t$ is joined to $L$ by two of the path's edges if $z_{t-1},z_{t+1}$ are both in $L$, by one if they lie on opposite sides, and by none if both lie in $R$. So mod 2, $\nu(y') - \nu(y)$ counts the visits to $M$ whose neighbors lie on opposite sides.

    Now read the symbols $L$ and $R$ off in order, skipping all occurrences of $M$. By (ii), the sequence changes between $L$ and $R$ exactly at those visits to $M$ whose neighbors lie on opposite sides. Since the path goes from vertex 1 to $n$, the sequence starts at $L$ and finishes at $R$. Hence $\nu(y') - \nu(y)$ is odd.
\end{proof}

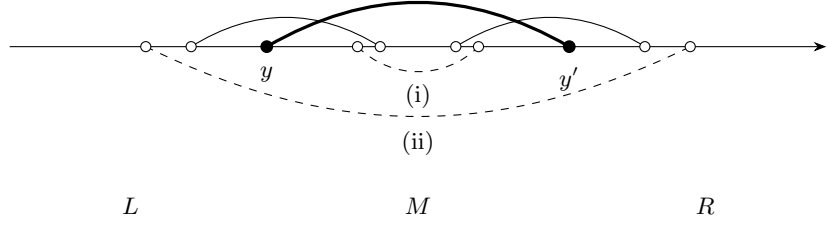
\begin{figure}[h!]
\centering
\begin{tikzpicture}[x=1cm,y=1cm,>=Stealth,font=\small]

  \draw[->] (-0.4,0) -- (10.4,0);
  \fill (3,0) circle (2.4pt); \node[below=4pt] at (3,0) {$y$};
  \fill (7,0) circle (2.4pt); \node[below=4pt] at (7,0) {$y'$};
  \node at (1.2,-2.1) {$L$}; \node at (5,-2.1) {$M$}; \node at (8.8,-2.1) {$R$};
  \draw[very thick] (3,0) to[bend left=30] (7,0);
  \draw[dashed] (4.2,0) to[bend right=45] node[below=1pt]{(i)} (5.8,0);
  \draw[dashed] (1.4,0) to[bend right=26] node[below=2pt,pos=0.5]{(ii)} (8.6,0);
  \draw (2.0,0) to[bend left=32] (4.5,0);
  \draw (5.5,0) to[bend left=32] (8.0,0);
  \foreach \x in {1.4,2.0,4.2,4.5,5.5,5.8,8.0,8.6}{\draw[fill=white] (\x,0) circle (1.9pt);}
\end{tikzpicture}
\caption{The two configurations ruled out by $G\in S(\ell,m,n)$. The thick arc is the edge of $Y$ between $y$ and $y'$. A path edge $e_t$ with both ends in $M$ would be nested inside $yy'$, giving (i), and one with an end in $L$ and an end in $R$ would have $yy'$ nested inside it, giving (ii). Edges of the path from $L$ to $M$ and from $M$ to $R$, drawn solid, are allowed.}
\label{fig:LMR}
\end{figure}

\begin{example}\label{ex:big}
    Let $n=11$ and let $G$ have the eleven edges
    \[
      \underbrace{(2,3)}_{f},\qquad
      \underbrace{(3,6),\,(3,7),\,(6,10),\,(7,10)}_{\text{edges of }Y},\qquad
      \underbrace{(1,2),\,(2,4),\,(4,8),\,(5,8),\,(5,9),\,(9,11)}_{\text{edges of }Z},
    \]
    shown in Figure~\ref{fig:big}. In lexicographic order the two rows are
    \[
      (1,2,2,3,3,4,5,5,6,7,9)\qquad\text{and}\qquad (2,3,4,6,7,8,8,9,10,10,11),
    \]
    both non-decreasing, and every degree is at most $3$, so $G\in S(4,3,11)$. It is
    connected, and $f=(2,3)$ is its only edge between $\{3,6,7,10\}$ and the rest, so
    $f$ is a cut-edge with
    \[
      Y=\{3,6,7,10\},\qquad Z=\{1,2,4,5,8,9,11\}\ni 1,11 .
    \]
    
    Take the path $1-2-4-8-5-9-11$ in $Z$, so
    \begin{gather*}
        (z_0,\dots,z_6)=(1,2,4,8,5,9,11) \\
      e_1=(1,2),\ e_2=(2,4),\ e_3=(4,8),\ e_4=(5,8),\ e_5=(5,9),\ e_6=(9,11),
    \end{gather*}
    which doubles back at $z_3=8$. Counting the edges of the path lying entirely below
    each vertex of $Y$ gives
    \[
      \nu(3)=1,\qquad \nu(6)=2,\qquad \nu(7)=2,\qquad \nu(10)=5,
    \]
    so $3$ and $10$ are {\color{red}red} and $6$ and $7$ are {\color{blue}blue}. Each of
    the four edges of $Y$ joins a red vertex to a blue one, and the coloring recovers
    the bipartition $\{3,10\}\sqcup\{6,7\}$ of the $4$-cycle.
    
    \smallskip
    \emph{Consider the edge $(3,7)$.} Here $L=\{1,2\}$, $M=\{4,5,6\}$, $R=\{8,\dots,11\}$
    and the path visits $L,L,M,R,M,R,R$. The visit $z_2=4$ has neighbors $z_1=2\in L$
    and $z_3=8\in R$ on opposite sides and contributes the single edge $e_2$ from $L$ to
    $M$; the visit $z_4=5$ has both neighbors in $R$ and contributes none. So
    $\nu(7)-\nu(3)=1$.
    
    \smallskip
    \emph{Consider the edge $(6,10)$.} Here $L=\{1,\dots,5\}$, $M=\{7,8,9\}$,
    $R=\{11\}$ and the path visits $L,L,L,M,L,M,R$. Now the visit $z_3=8$ has
    \emph{both} neighbors in $L$, so it contributes the two edges $e_3,e_4$ from $L$ to
    $M$ without changing the $L/R$ symbol; the visit $z_5=9$ has neighbors on opposite
    sides and contributes $e_5$. So $\nu(10)-\nu(6)=3$, odd, and deleting the $M$'s from
    $L,L,L,M,L,M,R$ leaves $L,L,L,L,R$, which changes letter exactly once.
\end{example}
    
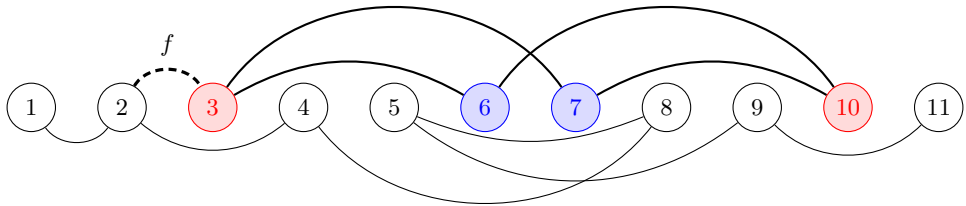
\begin{figure}[ht]
    \centering
    \begin{tikzpicture}[x=1.2cm,y=1cm,>=Stealth,font=\small,
        v/.style={circle,draw,inner sep=0pt,minimum size=6.4mm,fill=white}]
      \foreach \x in {1,2,4,5,8,9,11}{\node[v] (n\x) at (\x,0) {$\x$};}
      \node[v,draw=red,fill=red!14,text=red]   (n3)  at (3,0)  {$3$};
      \node[v,draw=blue,fill=blue!14,text=blue] (n6) at (6,0)  {$6$};
      \node[v,draw=blue,fill=blue!14,text=blue] (n7) at (7,0)  {$7$};
      \node[v,draw=red,fill=red!14,text=red]   (n10) at (10,0) {$10$};
      \draw[very thick,densely dashed] (n2) to[bend left=60] node[above=1pt]{$f$} (n3);
      \draw[thick] (n3) to[bend left=30]  (n6);
      \draw[thick] (n3) to[bend left=55]  (n7);
      \draw[thick] (n6) to[bend left=55]  (n10);
      \draw[thick] (n7) to[bend left=30]  (n10);
      \draw (n1) to[bend right=55] (n2);
      \draw (n2) to[bend right=40] (n4);
      \draw (n4) to[bend right=52] (n8);
      \draw (n5) to[bend right=22] (n8);
      \draw (n5) to[bend right=38] (n9);
      \draw (n9) to[bend right=45] (n11);
    \end{tikzpicture}
    \caption{The graph of Example~\ref{ex:big}. The cut-edge $f$ is dashed; the four
    edges of $Y$, forming the $4$-cycle $3-6-10-7-3$, are thick above the line; the six
    edges of the path $1-2-4-8-5-9-11$ inside $Z$ are below it. Vertices of $Y$ are
    colored by the parity of $\nu$: {\color{red}red} for odd, {\color{blue}blue} for
    even.}
    \label{fig:big}
\end{figure}

\begin{corollary}
\label{cor:path}
    Each tree in the bridge forest of an element of $S_{\reg}(m,n)$ is a path graph.
\end{corollary}
\begin{proof}
    The statement is trivial for $m=1$, so we assume that $m \geq 2$. We can assume that $G \in S_{\reg}(m,n)$ is connected since each component is isomorphic to an element of $S_{\reg}(m,n')$ for some $n' \leq n$. Suppose there is some branching in $T_G$. That is, there is a block $B$ and cut-edges $f_1,f_2,f_3$ adjacent to $B$. No matter where vertices $1$ and $n$ are located, one of the cut-edges $f_i$ will satisfy the conditions of \Cref{lem:bipartite}. So one of the components $Y$ of $G - f_i$ must be bipartite. Now, the vertex degrees in $Y$ are all $m$ except for the vertex adjacent to $f_i$ which has degree $m-1$ in $Y$. Suppose $Y$ were bipartite with parts $P,Q$, and suppose that $P$ is the part containing the vertex with degree $m-1$. Then we would have
    \[
        m|P| - 1 = \sum_{v \in P} \deg_Y v = \sum_{v \in Q} \deg_Y v = m|Q|,
    \]
    which is impossible for $m \geq 2$. Thus $T_G$ must be a path graph.
\end{proof}

\begin{theorem}
\label{thm:maxFac}
    For all $m \geq 3$, 
    \[
        S_{\max}(\ell,m,n) \subseteq S_{\max}(\ell,2,n) \cdot S_{\max}(\ell,m-2,n).
    \]
\end{theorem}
\begin{proof}
    For $G \in S_{\max}(\ell,m,n)$, there is an external regularization $R \in S_{\reg}(m,n')$ by \Cref{lem:maxreg}. By \Cref{cor:path} and \Cref{thm:2f}, $R$ has a 2-factor. The result then follows from \Cref{thm:extreg}, noting that the restriction of a 2-factor $F$ (resp. $(m-2)$-factor $R-F$) of $R$ to $G$ is contained in $S_{\max}(\ell,2,n)$ (resp. $S_{\max}(\ell,m-2,n)$) because of how $R$ was constructed: $F_G$ (resp. $G-F_G$) is a column suffix of $F$ (resp. $R-F$), and any column suffix of a maximal tableau is maximal.
\end{proof}

\subsection{Some additive relations}

For $a < b < c$, we have linear relations
\begin{align*}
    (x_a-x_b) + (x_b-x_c) &= (x_a-x_c) \\
    \Matching{3}{1/2} + \Matching{3}{2/3} &= \Matching{3}{1/3},
\end{align*}
and for $a < b < c < d$, we have quadratic relations
\begin{align*}
    (x_a-x_b)(x_c-x_d) + (x_a-x_d)(x_b-x_c) &= (x_a-x_c)(x_b-x_d) \\
    \Matching{4}{1/2,3/4} + \Matching{4}{1/4,2/3} &= \Matching{4}{1/3,2/4}
\end{align*}
The quadratic relation is nice because it keeps vertex degrees the same, so it can always be applied. The linear relation can only be applied if it is amenable to the vertex degrees in the global graph.

\begin{theorem}
\label{thm:degOne}
    $G(\ell,1,n) \cdot G(\ell,1,n)$ spans $G(\ell,2,n)$.
\end{theorem}
\begin{proof}
    An element of $G(\ell,2,n)$ is in $G(\ell,1,n) \cdot G(\ell,1,n)$ if and only if it contains no odd cycles. A 2-edge-coloring of a disjoint union of paths and even cycles can be chosen with each part of size $\leq \lceil E/2 \rceil \leq \ell$. 
    
    Suppose we have some $G \in G(\ell,2,n)$ with an odd cycle. We will show that $G$ is a linear combination of elements of $G(\ell,2,n)$ with no odd cycles. Applying quadratic relations, we can reduce to the case that there is one odd cycle and no edges not contained in the cycle (see \Cref{fig:relations}). Then there must be a vertex not contained in the cycle (if $n$ is even, this is clear; if $n$ is odd, we make use of the assumption that $\ell \leq n/2$). Then applying a linear relation finishes the proof (see \Cref{fig:relations}).
\end{proof}

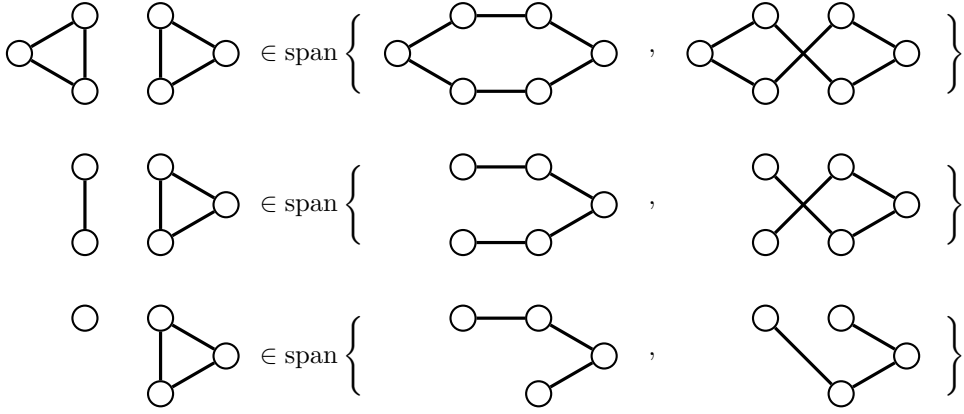
\begin{figure}[h!]
    \centering
    \begin{tikzpicture}
    \begin{scope}[every node/.style={circle,thick,draw}]
        \node (A1) at (-0.8660254,0){};
        \node (B1) at (0,.5){};
        \node (C1) at (0,-.5){};
        \node (D1) at (1,.5){};
        \node (E1) at (1,-.5){};
        \node (F1) at (1.8660254,0){};
    \end{scope}

    \begin{scope}[>={Stealth[black]},
                  every node/.style={fill=white},
                  every edge/.style={draw=black,very thick}]
        \path (A1) edge (B1);
        \path (B1) edge (C1);
        \path (A1) edge (C1);
        \path (D1) edge (E1);
        \path (E1) edge (F1);
        \path (D1) edge (F1);
    \end{scope}

    \node at (3,0) {$\in \operatorname{span} \Bigg\{$};

    \begin{scope}[every node/.style={circle,thick,draw}]
        \node (A2) at (5-.8660254,0){};
        \node (B2) at (5,.5){};
        \node (C2) at (5,-.5){};
        \node (D2) at (6,.5){};
        \node (E2) at (6,-.5){};
        \node (F2) at (6.8660254,0){};
    \end{scope}

    \node at (7.5,0) {,};

    \begin{scope}[>={Stealth[black]},
                  every node/.style={fill=white},
                  every edge/.style={draw=black,very thick}]
        \path (A2) edge (B2);
        \path (B2) edge (D2);
        \path (A2) edge (C2);
        \path (C2) edge (E2);
        \path (E2) edge (F2);
        \path (D2) edge (F2);
    \end{scope}

    \begin{scope}[every node/.style={circle,thick,draw}]
        \node (A) at (9-0.8660254,0){};
        \node (B) at (9,.5){};
        \node (C) at (9,-.5){};
        \node (D) at (10,.5){};
        \node (E) at (10,-.5){};
        \node (F) at (10.8660254,0){};
    \end{scope}

    \begin{scope}[>={Stealth[black]},
                  every node/.style={fill=white},
                  every edge/.style={draw=black,very thick}]
        \path (A) edge (B);
        \path (B) edge (E);
        \path (A) edge (C);
        \path (D) edge (C);
        \path (E) edge (F);
        \path (D) edge (F);
    \end{scope}

    \node at (11.5,0) {$\Bigg\}$};

    \begin{scope}[every node/.style={circle,thick,draw}]
        \node (B3) at (0,.5-2){};
        \node (C3) at (0,-.5-2){};
        \node (D3) at (1,.5-2){};
        \node (E3) at (1,-.5-2){};
        \node (F3) at (1.8660254,0-2){};
    \end{scope}

    \begin{scope}[>={Stealth[black]},
                  every node/.style={fill=white},
                  every edge/.style={draw=black,very thick}]
        \path (B3) edge (C3);
        \path (D3) edge (E3);
        \path (E3) edge (F3);
        \path (D3) edge (F3);
    \end{scope}

    \node at (3,0-2) {$\in \operatorname{span} \Bigg\{$};

    \begin{scope}[every node/.style={circle,thick,draw}]
        \node (B4) at (5,.5-2){};
        \node (C4) at (5,-.5-2){};
        \node (D4) at (6,.5-2){};
        \node (E4) at (6,-.5-2){};
        \node (F4) at (6.8660254,0-2){};
    \end{scope}

    \node at (7.5,0-2) {,};

    \begin{scope}[>={Stealth[black]},
                  every node/.style={fill=white},
                  every edge/.style={draw=black,very thick}]
        \path (B4) edge (D4);
        \path (C4) edge (E4);
        \path (E4) edge (F4);
        \path (D4) edge (F4);
    \end{scope}

    \begin{scope}[every node/.style={circle,thick,draw}]
        \node (B5) at (9,.5-2){};
        \node (C5) at (9,-.5-2){};
        \node (D5) at (10,.5-2){};
        \node (E5) at (10,-.5-2){};
        \node (F5) at (10.8660254,0-2){};
    \end{scope}

    \begin{scope}[>={Stealth[black]},
                  every node/.style={fill=white},
                  every edge/.style={draw=black,very thick}]
        \path (B5) edge (E5);
        \path (D5) edge (C5);
        \path (E5) edge (F5);
        \path (D5) edge (F5);
    \end{scope}

    \node at (11.5,0-2) {$\Bigg\}$};

    \begin{scope}[every node/.style={circle,thick,draw}]
        \node (B6) at (0,.5-4){};
        \node (D6) at (1,.5-4){};
        \node (E6) at (1,-.5-4){};
        \node (F6) at (1.8660254,0-4){};
    \end{scope}

    \begin{scope}[>={Stealth[black]},
                  every node/.style={fill=white},
                  every edge/.style={draw=black,very thick}]
        \path (D6) edge (E6);
        \path (E6) edge (F6);
        \path (D6) edge (F6);
    \end{scope}

    \node at (3,0-4) {$\in \operatorname{span} \Bigg\{$};

    \begin{scope}[every node/.style={circle,thick,draw}]
        \node (B7) at (5,.5-4){};
        \node (D7) at (6,.5-4){};
        \node (E7) at (6,-.5-4){};
        \node (F7) at (6.8660254,0-4){};
    \end{scope}

    \node at (7.5,0-4) {,};

    \begin{scope}[>={Stealth[black]},
                  every node/.style={fill=white},
                  every edge/.style={draw=black,very thick}]
        \path (B7) edge (D7);
        \path (E7) edge (F7);
        \path (D7) edge (F7);
    \end{scope}

    \begin{scope}[every node/.style={circle,thick,draw}]
        \node (B8) at (9,.5-4){};
        \node (D8) at (10,.5-4){};
        \node (E8) at (10,-.5-4){};
        \node (F8) at (10.8660254,0-4){};
    \end{scope}

    \begin{scope}[>={Stealth[black]},
                  every node/.style={fill=white},
                  every edge/.style={draw=black,very thick}]
        \path (B8) edge (E8);
        \path (E8) edge (F8);
        \path (D8) edge (F8);
    \end{scope}

    \node at (11.5,0-4) {$\Bigg\}$};
    
    \end{tikzpicture}
    \caption{The reduction in the proof of \Cref{thm:degOne}. The first row shows how two cycles can be combined into larger cycles using a quadratic relation. The second row shows how a cycle and an edge not in the cycle can be turned into paths using a quadratic relation. The third row shows how a cycle can be turned into paths using a linear relation if there is a vertex not contained in the cycle.}
    \label{fig:relations}
\end{figure}

We arrive at the main result of this section, which gives another proof of the rationality of $E_{\Delta_{\ell,n}}(t,q)$.

\thmgendegone*
\begin{proof}
    By \Cref{lem:basis}, $\bigcup_{m=1}^\infty S_{\max}(\ell,m,n)$ generates $\mathcal{H}_{\Delta_{\ell,n}}$. Then by \Cref{thm:maxFac}, $S_{\max}(\ell,1,n) \cup S_{\max}(\ell,2,n)$ generates $\mathcal{H}_{\Delta_{\ell,n}}$. Finally, by \Cref{thm:degOne}, $G(\ell,1,n)$, hence $S_{\max}(\ell,1,n)$, generates $\mathcal{H}_{\Delta_{\ell,n}}$. That is, $\mathcal{H}_{\Delta_{\ell,n}}$ is generated in $x_0$-degree 1.
\end{proof}

\section{Future directions}

We list a few avenues for future exploration.

\begin{itemize}
    \item There are two parts of \cite[Conjecture 5.5]{reinerrhoades} that we did not address for hypersimplices:
    \begin{itemize}
        \item $\mathcal{H}_{\Delta_{\ell,n}}$ is Cohen-Macaulay.
        \item The interior ideal $\overline{\mathcal{H}}_{\Delta_{\ell,n}}$ is the canonical module of $\mathcal{H}_{\Delta_{\ell,n}}$. We have that $\operatorname{int}(m\Delta_{\ell,n}) \cap \mathbb{Z}^n \cong \tope_{\ell m -n, m-2,n} \cap \mathbb{Z}^n$. So $\overline{\mathcal{H}}_{\Delta_{\ell,n}}$ is spanned by graphs in $G(\ell,m,n)$ with degrees $\leq m-2$ and number of edges $\leq \ell m - n$.
    \end{itemize}

    \item Hypersimplices are the matroid base polytopes of uniform matroids. Given recent developments studying the graded Ehrhart theory of polytopes arising from matroids \cite{crowley2026graded}, it would be interesting to find a unified, matroid-theoretic approach for matroid base polytopes.
    
    \item Computational data suggests that $\mathcal{H}_{\tope_{\ell,m,n}}$ is generated in $x_0$-degree 1. It may be possible to generalize our approach in \Cref{sec:harm} to prove this observation.

    \item It would also be interesting to determine a necessary and/or sufficient condition for $\mathcal{H}_\tope$ to be generated in $x_0$-degree 1. 
\end{itemize}

\appendix
\crefalias{section}{appendix}

\section{Factoring when $m$ is odd}
\label{app:modd}

Let $\hat\ell := \min\{\ell,\lceil n/2 \rceil\}$. Then
\[
    G(\ell,m,n) = G(\hat\ell , m ,n)
\]
because the number of edges in a graph with $n$ vertices and vertex degrees $\leq m$ is bounded by $mn/2 \leq m \lceil n/2 \rceil$.

The goal of this section is the following characterization:

\thmfullfact*
\begin{proof}
    Case (1) follows from \Cref{cor:even}.

    Suppose $m$ and $m'$ are both odd. If $\hat\ell = 0$, there are no edges, so equality is immediate. If $\hat\ell = 1$, then an element of $G(\ell,m + m',n)$ has at most $m + m'$ edges. We can split these edges into a set of size at most $m$ and a set of size at most $m'$, showing equality. If $\hat\ell \geq 2$, then the triangle with $(m + m')/2$ edges between every pair of vertices is an element of $G(\ell,m + m',n)$ which is not contained in $G(\ell,m,n) \cdot G(\ell,m',n)$. Any factorization forces $\deg_{G_1}(v) = m$ at all three vertices, but the degree sum $3m$ is odd, which is not possible. This proves case (2).

    For case (3), first suppose that $m < \frac{2\hat \ell - 3}{3}$, so $m \leq \frac{2\hat \ell - 5}{3}$. Then the graph in \Cref{fig:minimal} (possibly with some additional isolated vertices added) is an element of $G(\ell,m+m',n)$ that is not contained in $G(\ell,m,n) \cdot G(\ell,m',n)$. All vertices other than the central vertex $z$ have degree $m + m'$, so we must have $\deg_{G_1}(v) = m$ for all $v \neq z$. By parity, this fact forces $G_1$ to contain all $m + 1$ cut-edges to $z$. But then we have $\deg_{G_1}(z) = m+1 > m$, a contradiction. The number of edges in the graph in \Cref{fig:minimal} is 
    \[
        (m + 1)\cdot \frac{3(m + m') + 1}{2} = (m + m')\left(\frac{3(m+1)}{2} + \frac{m + 1}{2(m+m')} \right) \leq (m+m') \cdot \frac{3m + 5}{2} \leq \hat\ell(m + m')
    \]
    and the number of vertices is 
    \[
        3(m+1) + 1 = 3m + 4 \leq 2\hat\ell - 1 \leq n,
    \]
    so we do indeed get an element of $G(\ell,m + m',n)$.

    It remains to show that we have equality in case (3) when $m \geq \frac{2\hat\ell - 3}{3}$. It is sufficient to show that 
    \[
        G(\ell,2,n)\cdot G(\ell,M-2,n) = G(\ell,M,n)
    \]
    if $M - 2 \geq \frac{2\hat \ell - 3}{3}$ is odd. This will be proven in \Cref{thm:bound}.
\end{proof}

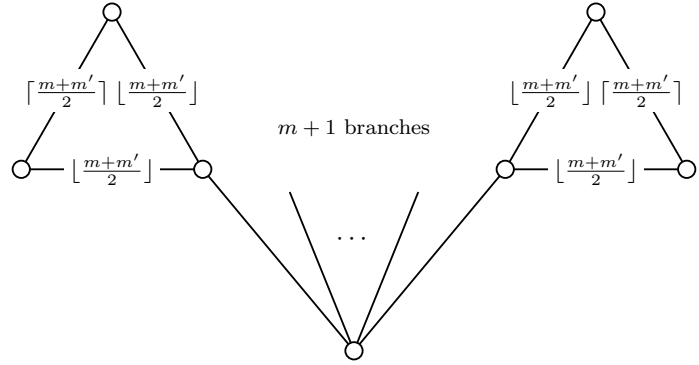
\begin{figure}[h!]
\centering
\begin{tikzpicture}
  \coordinate (c) at (0,-0.2);
  \coordinate (OL) at (-4.4,2.2); \coordinate (TL) at (-3.2,4.279); \coordinate (IL) at (-2.0,2.2);
  \coordinate (OR) at (4.4,2.2);  \coordinate (TR) at (3.2,4.279);  \coordinate (IR) at (2.0,2.2);
  \draw[e] (IL)--(c); \draw[e] (IR)--(c);
  \draw[e] (c)--(-0.85,1.9); \draw[e] (c)--(0.85,1.9);
  \node at (0,1.28){$\cdots$};
  \node[tag] at (0,2.75){$m+1$ branches};
  \draw[e] (OL)--(TL) node[mult]{$\lceil \frac{m + m'}{2}\rceil$};
  \draw[e] (TL)--(IL) node[mult]{$\lfloor \frac{m + m'}{2}\rfloor$};
  \draw[e] (OL)--(IL) node[mult]{$\lfloor \frac{m + m'}{2}\rfloor$};
  \draw[e] (OR)--(TR) node[mult]{$\lceil \frac{m + m'}{2}\rceil$};
  \draw[e] (TR)--(IR) node[mult]{$\lfloor \frac{m + m'}{2}\rfloor$};
  \draw[e] (OR)--(IR) node[mult]{$\lfloor \frac{m + m'}{2}\rfloor$};
  \foreach \p in {OL,TL,IL,OR,TR,IR,c} \node[vtx] at (\p){};
  \end{tikzpicture}
  \caption{An element of $G(\ell,m + m',n)$ that is not contained in $G(\ell,m,n) \cdot G(\ell,m',n)$ when $m < \frac{2\hat{\ell} - 3}{3}$ is odd and $m'$ is even.}
\label{fig:minimal}
\end{figure}

The following lemma is trivial but will allow us to make an important simplification later.

\begin{lemma}
\label{lem:addv}
    Consider the injection $\iota : G(\ell, m, n) \hookrightarrow G(\ell, m , n')$ for $n < n'$ obtained by adding $n' - n$ isolated vertices with labels $n + 1, \dots, n'$. We have
    \[
        \iota(G(\ell, 2,n)) \cdot \iota(G(\ell,m-2,n)) = \iota(G(\ell,m,n)) \cap (G(\ell,2,n') \cdot G(\ell,m-2,n')).
    \]
    In particular, to show that $G \in G(\ell,2,n) \cdot G(\ell,m-2,n)$, it suffices to show that $\iota(G)$ has an external regularization with a 2-factor for some $n' \geq n$.
\end{lemma}

\begin{proposition}
\label{thm:bound}
    If $m \geq \frac{2\hat\ell + 3}{3}$ is odd, then
    \[
        G(\ell,m,n) = G(\ell,2,n) \cdot G(\ell,m-2,n).
    \]
\end{proposition}
\begin{proof}
The bound $m \geq \frac{2\hat\ell + 3}{3}$ translates to a bound on the number of edges of $G$ as a function of $m$:
\[
    E(G) \leq \hat\ell m \leq \frac{3m^2 - 3m}{2}.
\]

We begin by considering what happens to the bridge forest when we add edges from an external vertex. Suppose for simplicity that $G$ is connected and that we add edges from only a single vertex. Say we add edges going into blocks $B_1$ and $B_2$ in $T_G$. These added edges have the effect of collapsing the unique path between $B_1$ and $B_2$ into a single block in the bridge tree of the new graph. So if we were to add edges from a single vertex into every block with positive deficit, the resulting bridge tree would have as its leaves the full leaf-blocks of $G$ and possibly one more if the identification of all positive deficit blocks creates a leaf-block (\Cref{fig:bridge-contraction}).

\begin{figure}[ht]
  \centering
  \begin{tikzpicture}[
      blk/.style = {circle, draw, minimum size=6.5mm, inner sep=0pt, font=\small}, 
      ext/.style = {circle, draw, fill=black, minimum size=4pt, inner sep=0pt},     
      wl/.style  = {font=\small, inner sep=2pt},
      br/.style  = {thick},        
      nw/.style  = {dashed},       
      enc/.style = {draw, dashed, rounded corners, inner sep=5pt},
      tt/.style  = {font=\small},
    ]
    \begin{scope}
      \node[tt] at (0,2.3) {(A) collapsing the $B_1$--$B_2$ path};
      \node[blk] (B1) at (-1.95,0) {$B_1$};
      \node[blk] (a)  at (-0.65,0) {};
      \node[blk] (b)  at (0.65,0)  {};
      \node[blk] (B2) at (1.95,0)  {$B_2$};
      \node[blk] (c2) at (-1.0,1.25) {};
      \node[blk] (c3) at (1.0,1.25)  {};
      \draw[br] (B1)--(a)--(b)--(B2);
      \draw[br] (a)--(c2); \draw[br] (b)--(c3);
      \node[ext] (w) at (0,-1.55) {};
      \draw[nw] (w)--(B1); \draw[nw] (w)--(B2);
      \node[enc, fit=(B1)(a)(b)(B2)(w)] (box) {};
    \end{scope}
    \begin{scope}[shift={(8,0)}]
      \node[tt] at (0.1,2.3) {(B) the merged block is a leaf};
      \node[blk] (H)  at (0,0)       {};
      \node[blk] (H2) at (-1.4,0.3)  {};
      \node[blk] (L1) at (-2.5,0.95) {};
      \node[blk] (L2) at (-2.5,-0.5) {};
      \node[blk] (L3) at (-0.2,-1.4) {};
      \node[blk] (Q1) at (1.4,0)     {};
      \node[blk] (Q2) at (2.4,0.85)  {};
      \node[blk] (Q3) at (2.4,-0.85) {};
      \draw[br] (H)--(H2); \draw[br] (H2)--(L1); \draw[br] (H2)--(L2); \draw[br] (H)--(L3);
      \draw[br] (H)--(Q1); \draw[br] (Q1)--(Q2); \draw[br] (Q1)--(Q3);
      \node[ext] (w) at (3.1,0) {};
      \draw[nw] (w)--(Q2); \draw[nw] (w)--(Q3);
      \node[enc, fit=(Q1)(Q2)(Q3)(w)] (box2) {};
    \end{scope}
  \end{tikzpicture}
  \caption{The effect on the bridge forest of adding edges from a single external vertex. Circles are blocks of $G$, the lines between them are cut-edges, the solid dot is the external vertex, and the dashed box encloses the blocks that merge into a single block of the new graph. (A) The entire $B_1$--$B_2$ path is collapsed into one block. (B) When every block of positive deficit is hit, the merged block can itself be a leaf of the new bridge tree, causing a ``$+1$'' in the leaf-block count.}
  \label{fig:bridge-contraction}
\end{figure}
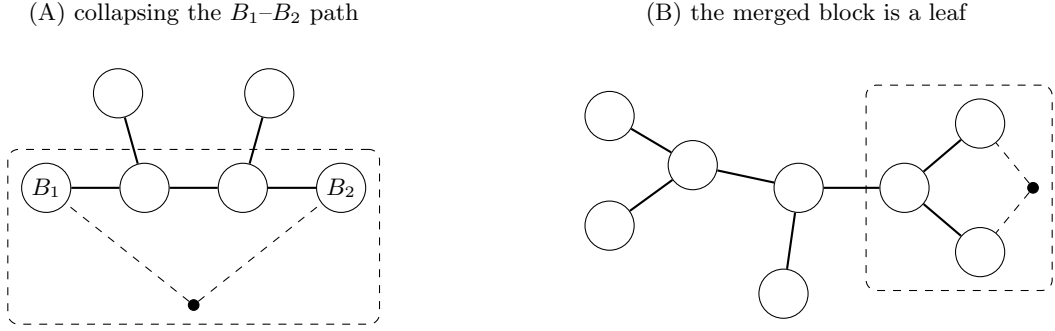

But for an external regularization, we are adding edges from multiple external vertices, and we must be careful about how we partition the total deficit among these vertices. Still assuming that $G$ is connected, each added vertex becomes part of the same block as all of the vertices of $G$ it connects to, as well as all paths between these vertices (except in the case when the last external vertex has only one edge into $G$ --- when the total deficit is $1 \mod m$).

Since we want to minimize the degrees of blocks in $T_R$, we begin by adding edges to leaf-blocks of $G$, as this will tend to collapse the most edges in $T_G$. Moreover, a block of $R$ containing an external vertex can only be a leaf-block if it also contains a leaf-block of $G$ (except possibly when the total deficit is $1 \mod m$). So by taking care of the leaf-blocks of $G$ first, we create at most $\lceil L_{+}/m \rceil + 1$ leaf-blocks in $R$, where $L_+$ is the number of leaf-blocks of $G$ with positive deficit and the $+1$ comes from the case that the total deficit is $\pm 1 \mod m$. 

But we can do better. If a leaf-block of $G$ has deficit $\geq 2$, then we can split up the added edges connecting to it across multiple external vertices. These external vertices then become part of the same block of $R$, so the number of possible leaf-blocks decreases. Hence we create at most $\lceil \#AFL(G)/m \rceil + 1$ leaf-blocks in $R$. The +1 comes from at most one of the following cases:
\begin{itemize}
    \item If the total deficit is $\pm 1 \mod m$, then the process in \Cref{fig:cases} adds a leaf-block.

    \item If $\#AFL(G) = 0$ and the identification of all blocks with positive deficit creates a leaf-block.
\end{itemize}
Suppose for contradiction that both cases occur: let $B$ be the leaf-block in $T_R$ arising from the second case. In order for $B$ to be a leaf-block, we must have a leaf-block with deficit $\geq 2$ in $T_G$ (using that $\#AFL(G) = 0$).
\begin{itemize}
        \item If the total deficit is $1 \mod m$, then we ensure that the leaf-block from \Cref{fig:cases} attaches to a leaf-block of $G$ with deficit $\geq 2$ (before attaching any other edges). This makes it so that $B$ is not a leaf-block. 

        \item If the total deficit is $-1 \mod m$, then the leaf-block from \Cref{fig:cases} already attaches to $B$, so $B$ cannot be a leaf-block.
    \end{itemize}

We now suppose that $G$ is not connected. Let $G_0 = \bigcup_{C \in C_0(G)} C$ be the union of the connected components of $G$ with total deficit 0, and let $G_+ = G \setminus G_0$. We begin by making $G_+$ connected by adding an edge from an external vertex to each connected component of $G$ with positive deficit. If there are more than $m$ components, then we continue with a second external vertex using a component with total deficit $\geq 2$ to bridge between them, and so on. This produces a connected graph because our bound on the edges forces there to be at most $m-1$ components with total deficit $1$ (a component with deficit $1$ contains at least $(3m-1)/2$ edges). If the number of edges added in this connection process is not a multiple of $m$, then we add additional vertices to $G$ and use \Cref{lem:addv}.

An added edge is a cut-edge iff it is the unique edge between an external vertex and a connected component of $G$. So all of our added edges at this point are cut-edges. It follows that we only create new full leaf-blocks (resp. almost full leaf-blocks) during this connection process when we connect to a component consisting of a single block with deficit 1 (resp. 2). We will denote the set of such components of $G$ as $C'_1(G)$ and $C'_2(G)$, respectively.

We then add the remaining edges to $G_+$ as we did for a connected graph. Then there are at most
\[
    \#FL(G) + \#C_1'(G) + \left\lceil \frac{\#AFL(G) + \#C_2'(G) }{m} \right\rceil + 1
\]
leaf-blocks in the resulting regular graph. 

The next step is to compute how many edges are necessary for each of the involved objects. Elements of $FL(G)$ and $C_1'(G)$ must contain at least $(3m-1)/2$ edges, and elements of $AFL(G)$ and $C_2'(G)$ must contain at least $m-1$ edges (we could get better bounds with a more careful analysis, but these will suffice for our purposes). 

Let $\alpha = \#FL(G) + \#C_1'(G)$ and $\beta = \#AFL(G) + \#C_2'(G)$. If the total number of edges in $G$ is $\leq 3m(m-1)/2$ and the external regularization procedure above produces at least $m$ leaf-blocks, then we must have
\[
    m - \lceil \beta/m\rceil - 1 \leq \alpha \leq \frac{(3m-2\beta)(m-1)}{3m-1}.
\]
For $m \geq 3$, the only non-negative integer solutions are $(\alpha,\beta) \in \{(m-1,0), (m-2,1), (m-2,2)\}$ and $(0,4)$ only in the case $m=3$. In the cases where $\beta \neq 0$, we must have that the total deficit is $\pm 1 \mod m$ justifying the +1 in the leaf-block count.

\begin{itemize}
    \item $\alpha = m-1$, $\beta = 0$. In this case we know that
    \[
        \frac{3m^2 - 4m + 1}{2} \leq E(G) \leq \frac{3m^2 - 3m}{2}.
    \]
    So $G$ is the graph in \Cref{fig:disjoint} possibly with additional vertices and at most $\frac{m-1}{2}$ additional edges. In particular, since $\frac{m-1}{2} \leq m-2$, all other vertices must have degree $\leq m-2$. The graph in \Cref{fig:disjoint} truly has no external regularization with a 2-factor. 
    
    We must have $\hat\ell = \frac{3m-3}{2}$. Consider the 2-factor $F$ consisting of $m-1$ disjoint triangles in $\alpha$. Then $G \setminus F$ has vertex degrees $\leq m-2$ and
    \begin{align*}
        E(F) &= 3m-3 = 2\hat\ell \\ 
        E(G \setminus F) &\leq \frac{3m^2-9m + 6}{2} = \hat\ell(m-2).
    \end{align*}

    \item $\alpha = m-2$, $\beta = 1$. The number of edges must be 
    \[
        \frac{3m^2 - 4m \pm 1}{2}
    \]
    (using that the total deficit is $\pm 1 \mod m$). So in addition to the 
    \[
        (m-2)\frac{3m-1}{2} + m-1 = \frac{3m^2-5m}{2}
    \]
    necessary edges for $\alpha$ and $\beta$, we have
    \[
        \frac{m \pm 1}{2}
    \]
    additional edges. In particular, for $m > 3$ the only vertices with degree $\geq m-1$ are those in the subgraph consisting of the necessary edges for $\alpha$ and $\beta$. Also we must have $\hat\ell = \frac{3m-3}{2}$. For such a graph, take $F$ to be the subgraph consisting of $m-2$ disjoint triangles in $\alpha$ and 2 of the edges from $\beta$. Then $G \setminus F$ has vertex degrees $\leq m-2$ and
    \begin{align*}
        E(F) &= 3m - 4 \leq 2\hat\ell \\
        E(G \setminus F) &= \frac{3m^2 - 10m + 8 \pm 1}{2} \leq \frac{3m^2 - 9m + 6}{2} \leq \hat\ell (m-2)
    \end{align*}

    For $m = 3$, it remains to consider the case where there are $2$ additional edges connected to a vertex not in $\alpha$ or $\beta$. There are four cases shown in \Cref{fig:m=3}. We include the case $\beta=2$ as well. A choice of $F$ is shown in red. One can check that $E(F) \leq 2\hat\ell = 6$ and $E(G \setminus F) \leq \hat\ell = 3$.

    \item $\alpha = m-2$, $\beta = 2$. Considering the number of edges, there is only one such $G$ (up to adding vertices) which has $m=3$ and was addressed in the previous case.

    \item $\alpha = 0$, $\beta = 4$. There is only one graph to consider consisting of four disjoint digons. There is no issue in constructing an external regularization for this graph.
\end{itemize}
\end{proof}

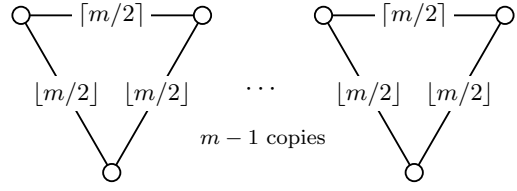
\begin{figure}[h!]
        \centering
        \begin{tikzpicture}
      \foreach \x in {-2.0,2.0}{
        \begin{scope}[xshift=\x cm]
          \coordinate (a) at (-1.2,2.079); \coordinate (b) at (1.2,2.079); \coordinate (d) at (0,0);
          \draw[e] (a)--(b) node[mult]{$\lceil m/2\rceil$};
          \draw[e] (a)--(d) node[mult]{$\lfloor m/2\rfloor$};
          \draw[e] (b)--(d) node[mult]{$\lfloor m/2\rfloor$};
          \foreach \p in {a,b,d} \node[vtx] at (\p){};
        \end{scope}
      }
      \node at (0,1.1){$\cdots$};
      \node[tag] at (0,0.45){$m-1$ copies};
    \end{tikzpicture}
    \caption{The graph in the case $\alpha = m-1$, $\beta = 0$ (recall $m$ is odd): $m-1$ disjoint copies of the triangle with edge multiplicities $\lceil m/2 \rceil, \lfloor m/2 \rfloor, \lfloor m/2 \rfloor$. Each component is a single block with $(3m-1)/2$ edges and has deficit $1$. Any external regularization therefore has at least $m$ leaf-blocks and, by \Cref{thm:2f}, no $2$-factor. The $2$-factor $F$ of the proof takes the triangle from each copy.}
    \label{fig:disjoint}
    \end{figure}

    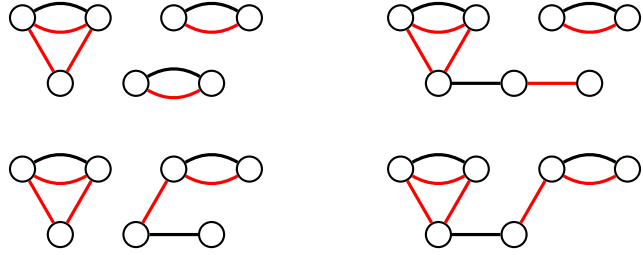
\begin{figure}[h!]
        \centering
        \begin{tikzpicture}
    \begin{scope}[every node/.style={circle,thick,draw}]
        \node (A1) at (0,.866) {};
        \node (B1) at (1,.866) {};
        \node (C1) at (.5,0) {};
        \node (D1) at (1.5,0) {};
        \node (E1) at (2,.866) {};
        \node (F1) at (3,.866) {};
        \node (G1) at (2.5,0) {};
    \end{scope}

    \begin{scope}[every node/.style={circle,thick,draw}]
        \node (A2) at (0+5,.866) {};
        \node (B2) at (1+5,.866) {};
        \node (C2) at (.5+5,0) {};
        \node (D2) at (1.5+5,0) {};
        \node (E2) at (2+5,.866) {};
        \node (F2) at (3+5,.866) {};
        \node (G2) at (2.5+5,0) {};
    \end{scope}

    \begin{scope}[every node/.style={circle,thick,draw}]
        \node (A3) at (0,.866-2) {};
        \node (B3) at (1,.866-2) {};
        \node (C3) at (.5,0-2) {};
        \node (D3) at (1.5,0-2) {};
        \node (E3) at (2,.866-2) {};
        \node (F3) at (3,.866-2) {};
        \node (G3) at (2.5,0-2) {};
    \end{scope}

    \begin{scope}[every node/.style={circle,thick,draw}]
        \node (A4) at (0+5,.866-2) {};
        \node (B4) at (1+5,.866-2) {};
        \node (C4) at (.5+5,0-2) {};
        \node (D4) at (1.5+5,0-2) {};
        \node (E4) at (2+5,.866-2) {};
        \node (F4) at (3+5,.866-2) {};
    \end{scope}

    \begin{scope}[>={Stealth[red]},
                  every node/.style={fill=white},
                  every edge/.style={draw=red,very thick}]
        \path[bend right] (A1) edge (B1);
        \path (A1) edge (C1);
        \path (B1) edge (C1);
        \path[bend right] (E1) edge (F1);
        \path[bend right] (D1) edge (G1);

        \path[bend right] (A2) edge (B2);
        \path (A2) edge (C2);
        \path (B2) edge (C2);
        \path[bend right] (E2) edge (F2);
        \path (D2) edge (G2);

        \path[bend right] (A3) edge (B3);
        \path (A3) edge (C3);
        \path (B3) edge (C3);
        \path[bend right] (E3) edge (F3);
        \path (D3) edge (E3);

        \path[bend right] (A4) edge (B4);
        \path (A4) edge (C4);
        \path (B4) edge (C4);
        \path[bend right] (E4) edge (F4);
        \path (D4) edge (E4);
    \end{scope}

    \begin{scope}[>={Stealth[black]},
                  every node/.style={fill=white},
                  every edge/.style={draw=black,very thick}]
        \path[bend left] (A1) edge (B1);
        \path[bend left] (E1) edge (F1);
        \path[bend left] (D1) edge (G1);
        
        \path[bend left] (A2) edge (B2);
        \path[bend left] (E2) edge (F2);
        \path (D2) edge (C2);

        \path[bend left] (A3) edge (B3);
        \path[bend left] (E3) edge (F3);
        \path (D3) edge (G3);
 
        \path[bend left] (A4) edge (B4);
        \path[bend left] (E4) edge (F4);
        \path (D4) edge (C4);
    \end{scope}
    
    \end{tikzpicture}
    \caption{The four remaining configurations when $m=3$: the block from $\alpha$ with $(3m-1)/2 = 4$ edges, the digon(s) from $\beta$, and two additional edges meeting a vertex outside $\alpha$ and $\beta$. The additional edges form their own digon (top left, which also covers $\beta = 2$), hang off the block in $\alpha$ (top right), hang off the digon in $\beta$ (bottom left), or join the two (bottom right). A choice of $F$ is drawn in red.}
    \label{fig:m=3}
    \end{figure}

\begin{remark}
    The above construction is a little awkward in its handling of multiple connected components. One would like to simply construct external regularizations for each component separately, then take a 2-factor in each component. The issue that arises here is with the $\ell$ parameter. Consider the graph in $G(2,3,4)$ pictured in \Cref{fig:discon}. It has two possible choices of factors in $G(2,2,4)$ -- one with 3 edges and one with 4 edges. But an external regularization can only detect the one with 3 edges.

    Now consider the graph in $G(5,3,12)$ that is 3 disjoint copies of the previous graph. Because of the restrictions with $\ell = 5$, our factor in $G(5,2,12)$ must have exactly 10 edges. So we need to choose the factor with 3 edges in two of the components and the factor with 4 edges in the third. We have shown in \Cref{fig:disconfac} how our process above detects this.
\end{remark}

\begin{figure}[h!]
    \centering
    \begin{tikzpicture}
    \begin{scope}[every node/.style={circle,thick,draw}]
        \node (A1) at (1,1) {};
        \node (B1) at (0,0) {};
        \node (C1) at (2,0) {};
        \node (D1) at (1,-1) {};
        \node (A2) at (4,1) {};
        \node (B2) at (3,0) {};
        \node (C2) at (5,0) {};
        \node (D2) at (4,-1) {};
    \end{scope}

    \begin{scope}[>={Stealth[black]},
                  every node/.style={fill=white},
                  every edge/.style={draw=black,very thick}]
        \path (A1) edge (C1);
        \path (B1) edge (D1);
        \path (B2) edge (C2);
    \end{scope}
    \begin{scope}[>={Stealth[black]},
                  every node/.style={fill=white},
                  every edge/.style={draw=red,very thick}]
        \path (A1) edge (B1);
        \path (B1) edge (C1);
        \path (C1) edge (D1);
        \path (A2) edge (B2);
        \path (A2) edge (C2);
        \path (B2) edge (D2);
        \path (C2) edge (D2);
    \end{scope}
    
    \end{tikzpicture}
    \caption{An element of $G(2,3,4)$: a $4$-cycle with one diagonal. Its two factors in $G(2,2,4)$ are drawn in red: the $3$-edge path (left) and the $4$-cycle (right). Restricting a $2$-factor of an external regularization detects only the first.}
    \label{fig:discon}
\end{figure}
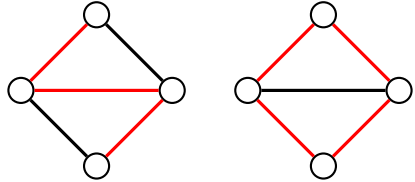

\begin{figure}[h!]
    \centering
    \begin{tikzpicture}
    \begin{scope}[every node/.style={circle,thick,draw}]
        \node (A1) at (1,1) {};
        \node (B1) at (0,0) {};
        \node (C1) at (2,0) {};
        \node (D1) at (1,-1) {};
        \node (A2) at (4,1) {};
        \node (B2) at (3,0) {};
        \node (C2) at (5,0) {};
        \node (D2) at (4,-1) {};
        \node (A3) at (7,1) {};
        \node (B3) at (6,0) {};
        \node (C3) at (8,0) {};
        \node (D3) at (7,-1) {};
        \node (E) at (4,2) {};
        \node (F) at (4,-2) {};
    \end{scope}

    \begin{scope}[>={Stealth[black]},
                  every node/.style={fill=white},
                  every edge/.style={draw=black,very thick}]
        \path (A1) edge (C1);
        \path (B1) edge (D1);
        \path (B2) edge (C2);
        \path (A3) edge (C3);
        \path (B3) edge (D3);
        \path (E) edge (A2);
        \path (F) edge (D2);
    \end{scope}
    \begin{scope}[>={Stealth[black]},
                  every node/.style={fill=white},
                  every edge/.style={draw=red,very thick}]
        \path (A1) edge (B1);
        \path (B1) edge (C1);
        \path (C1) edge (D1);
        \path (A2) edge (B2);
        \path (A2) edge (C2);
        \path (B2) edge (D2);
        \path (C2) edge (D2);
        \path (A3) edge (B3);
        \path (B3) edge (C3);
        \path (C3) edge (D3);
        \path (E) edge (A1);
        \path (F) edge (D1);
        \path (E) edge (A3);
        \path (F) edge (D3);
    \end{scope}
    
    \end{tikzpicture}
    \caption{The external regularization of the disjoint union of three copies of the graph in \Cref{fig:discon}, viewed in $G(5,3,12)$. The top and bottom external vertices each send one edge into every copy. A $2$-factor of the regularization is drawn in red. Its restriction takes the $4$-edge factor in the middle copy and the $3$-edge factor in the outer two, for the $10$ edges that $\ell = 5$ demands.}
    \label{fig:disconfac}
\end{figure}
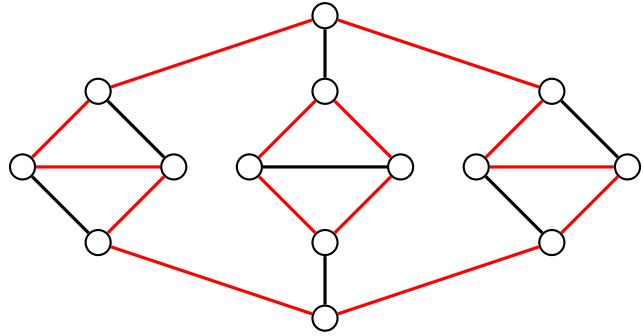

\begin{remark}
    It is worth noting that there are two types of obstructions that can occur when trying to factor an element of $G(\ell,m,n)$. The first is a structural obstruction coming from cut-edges like the one in \Cref{fig:minimal}. The second occurs when the $\ell$ parameter is too tight. Consider for example the graph in \Cref{fig:disjoint} but with $m$ copies instead of $m-1$. If we consider this as an element of $G\left(\frac{3m-1}{2} , m, 3m \right)$, then it does not have a factorization because the only possible 2-factor has $3m > 3m-1$ edges. The construction above deals primarily with obstructions of the first kind because these are encountered first as one increases the number of edges.
\end{remark}

\bibliographystyle{amsalpha}
\bibliography{bib}

\end{document}